\documentclass[11pt,reqno]{amsart}

\usepackage{amsmath,amssymb,mathrsfs,amsthm,mathabx}
\usepackage{graphicx,cite,cases}
\newtheorem{theorem}{Theorem}[section]

\newtheorem{lemma}[theorem]{Lemma}

\numberwithin{equation}{section}

\begin{document}

\title[convergence of an adaptive finite element DtN method]{Convergence of an
adaptive finite element DtN method for the elastic wave scattering problem}

\author{Peijun Li}
\address{Department of Mathematics, Purdue University, West Lafayette, Indiana
47907, USA.}
\email{lipeijun@math.purdue.edu}

\author{Xiaokai Yuan}
\address{Department of Mathematics, Purdue University, West Lafayette, Indiana
47907, USA.}
\email{yuan170@math.purdue.edu}


\subjclass[2010]{65N12, 65N15, 65N30, 78A45}

\keywords{Elastic wave equation, obstacle scattering problem, adaptive finite
element method, Dirichlet-to-Neumann map, transparent boundary condition, a
posteriori error estimate.}

\begin{abstract}
Consider the scattering of an elastic plane wave by a rigid obstacle, which is
immersed in a homogeneous and isotropic elastic medium in two dimensions.
Based on a Dirichlet-to-Neumann (DtN) operator, an exact transparent boundary
condition is introduced and the scattering problem is formulated as a boundary
value problem of the elastic wave equation in a bounded domain. By developing a
new duality argument, an a posteriori error estimate is derived for the discrete
problem by using the finite element method with the truncated DtN operator. The
a posteriori error estimate consists of the finite element approximation error
and the truncation error of the DtN operator which decays exponentially with
respect to the truncation parameter. An adaptive finite element algorithm is
proposed to solve the elastic obstacle scattering problem, where the truncation
parameter is determined through the truncation error and the mesh elements for
local refinements are chosen through the finite element discretization error.
Numerical experiments are presented to demonstrate the effectiveness of the
proposed method.
\end{abstract}

\maketitle

\section{Introduction}

A basic problem in classical scattering theory is the scattering of
time-harmonic waves by a bounded and impenetrable medium, which is known as the
obstacle scattering problem. It has played a crucial role in diverse
scientific areas such as radar and sonar, geophysical exploration, medical
imaging, and nondestructive testing. Motivated by these significant
applications, the obstacle scattering problem has been widely studied for
acoustic and electromagnetic waves. Consequently, a great deal of results are
available concerning its solution \cite{CK-83, M-03, N-01}. Recently, the
scattering
problems for elastic waves have received ever-increasing attention due to the
important applications in seismology and geophysics \cite{BHSY-jmpa18,
LWWZ-ip16, LY-ipi}. For instance, they are fundamental to detect the fractures
in sedimentary rocks for the production of underground gas and liquids. Compared
with acoustic and electromagnetic waves, elastic waves are less studied due to
the coexistence of compressional waves and shear waves that have different
wavenumbers \cite{C-88, LL-86}. 

The obstacle scattering problem is usually formulated as an exterior boundary
value problem imposed in an open domain. The unbounded physical domain needs to
be truncated into a bounded computational domain for the convenience of
mathematical analysis or numerical computation. Therefore, an appropriate
boundary condition is required on the boundary of the truncated domain to
avoid artificial wave reflection. Such a boundary condition is called the
transparent boundary condition (TBC) or non-reflecting boundary condition. It
is one of the important and active subjects in the research area of wave
propagation \cite{BT-cpam80, EM-mc77, GG-wm03, GK-wm90, GK-jcp95,
MK-jcp04, T-an99}. Since Berenger proposed a perfectly matched layer (PML)
technique to solve the time-dependent Maxwell equations \cite{B-jcp94}, the
research on the PML has undergone a tremendous development due to its
effectiveness and simplicity. Various constructions of PML have been proposed
and studied for a wide range of scattering problems on acoustic and
electromagnetic wave propagation \cite{BW-sinum05, CC-mc08, CW-motl94,
HSZ-sima03, TY-anm98, JL-cicp17}. The basic idea of the PML technique is to
surround the domain of interest by a layer of finite
thickness fictitious medium that attenuates the waves coming from inside of the
computational domain. When the waves reach the outer boundary of the PML region,
their values are so small that the homogeneous Dirichlet boundary conditions can
be imposed. 

A posteriori error estimates are computable quantities which measure the
solution errors of discrete problems. They are essential
in designing algorithms for mesh modification which aim to equidistribute the
computational effort and optimize the computation. The a posteriori
error estimates based adaptive finite element methods have the ability of error
control and asymptotically optimal approximation property \cite{BW-sinum78}.
They have become a class of important numerical tools for solving differential
equations, especially for those where the solutions have singularity or
multiscale phenomena. Combined with the PML technique, an efficient adaptive
finite element method was developed in \cite{CW-sinum03} for solving the
two-dimensional diffraction grating problem, where the medium has a
one-dimensional periodic
structure and the model equation is the two-dimensional Helmholtz equation. It
was shown that the a posteriori error estimate consists of the finite element
discretization error and the PML truncation error which decays exponentially
with respect to the PML parameters such as the thickness of the layer and the
medium properties. Due to the superior numerical performance, the adaptive PML
method was quickly extended to solve the two- and three-dimensional obstacle
scattering problems \cite{CC-mc08, CL-sinum05} and the three-dimensional
diffraction grating problem \cite{BLW-mc10}, where either the two-dimensional
Helmholtz equation or the three-dimensional Maxwell equations were considered.
Although the PML method has been developed to solve various elastic wave
propagation problems in engineering and geophysics soon after it was introduced
\cite{CL-jca96, CT-g01, HSB-jasa96, KT-gji03}, the rigorous mathematical studies
were only recently done for elastic waves because of the complex of the model
equation \cite{BP-mc10, CXZ-mc16, JLLZ-m2na17, JLLZ-cms18}. 

As a viable alternative, the finite element DtN method has been proposed to
solve the obstacle scattering problems \cite{JLLZ-jsc17, JLZ-cicp13} and the
diffraction grating problems \cite{WBLLW-sinum15, JLWWZ-18},
respectively, where the transparent boundary conditions are used to truncate
the domains. In this new approach, the layer of artificial medium is not needed
to enclose the domain of interest, which makes is different from the PML method.
The transparent boundary conditions are based on nonlocal Dirichlet-to-Neumann
(DtN) operators and are given as infinite Fourier series. Since the transparent
boundary conditions are exact, the artificial boundary can be put as close as
possible to the scattering structures, which can reduce the size of the
computational domain. Numerically, the infinite series need to be truncated
into a sum of finitely many terms by choosing an appropriate truncation
parameter $N$. It is known that the convergence of the truncated DtN map 
could be arbitrarily slow to the original DtN map in the operator norm. The a
posteriori error analysis of the PML method cannot be applied directly to the
DtN method since the DtN map of the truncated PML problem converges
exponentially fast to the DtN map of the untruncated PML problem. To overcome
this issue, a duality argument had to be developed to obtain the a posteriori
error estimate between the solution of the scattering problem and the finite
element solution. Comparably, the a posteriori error estimates consists of the
finite element discretization error and the DtN truncation error, which decays
exponentially with respect to the truncation parameter $N$. The numerical
experiments demonstrate that the adaptive DtN method has a competitive behavior
to the adaptive PML method. 

In this paper, we present an adaptive finite element DtN method and carry out
its mathematical analysis for the elastic wave scattering problem. The goal is
threefold: (1) prove the exponential convergence of the truncated DtN operator;
(2) give a complete a posteriori error estimate; (3) develop an effective
adaptive finite element algorithm. This paper significantly extends the work on
the acoustic scattering problem \cite{JLLZ-jsc17}, where the Helmholtz
equation was considered. Apparently, the techniques differ greatly from the
existing work because of the complicated transparent boundary condition
associated with the elastic wave equation. 

Specifically, we consider a rigid obstacle which is immersed in a homogeneous
and isotropic elastic medium in two dimensions. The Helmholtz decomposition is
utilized to formulate the exterior boundary value problem of the elastic wave
equation into a coupled exterior boundary value problem of the Helmholtz
equation. By using a Dirichlet-to-Neumann (DtN) operator, an exact transparent
boundary condition, which is given as a Fourier series, is introduced
to reduce the original scattering problem into a boundary value problem of
the elastic wave equation in a bounded domain. The discrete problem is studied
by using the finite element method with the truncated DtN operator. Based on
the Helmholtz decomposition, a new duality argument is developed to obtain an a
posteriori error estimate between the solution of the original scattering
problem and the discrete problem. The a posteriori error estimate consists of
the finite element approximation error and the truncation error of the DtN
operator which is shown to decay exponentially with respect to the truncation
parameter. The estimate is used to design the adaptive finite element algorithm
to choose elements for refinements and to determine the truncation parameter
$N$. Since the truncation error decays exponentially with respect to $N$, the
choice of the truncation parameter $N$ is not sensitive to the given tolerance. 
Numerical experiments are presented to demonstrate the effectiveness of the
proposed method.

The paper is organized as follows. In Section \ref{Section_pf}, the elastic wave
equation is introduced for the scattering by a rigid obstacle; a boundary value
problem is formulated by using the transparent boundary condition; the
corresponding weak formulation is discussed. In Section \ref{Section_dp}, the
discrete problem is considered by using the finite element approximation with
the truncated DtN operator. Section \ref{Section_pea} is devoted to the a
posteriori error analysis and serves as the basis of the adaptive algorithm. In
Section \ref{Section_ne}, we discuss the numerical implementation of the 
adaptive algorithm and present two numerical examples to illustrate the
performance of the proposed method. The paper is concluded with
some general remarks and directions for future work in Section \ref{Section_c}.

\section{Problem formulation}\label{Section_pf}

Consider a two-dimensional elastically rigid obstacle $D$ with Lipschitz
continuous boundary $\partial D$, as seen in Figure \ref{fig1}. Denote by $\nu$
and $\tau$ the unit normal and tangent vectors on $\partial D$, respectively.
The exterior domain $\mathbb R^2\setminus\overline D$ is assumed to be filled
with a homogeneous and isotropic elastic medium with a unit mass density. Let
$B_R=\{\boldsymbol x=(x, y)^\top\in\mathbb R^2: |\boldsymbol x|<R\}$ and
$B_{\hat R}=\{\boldsymbol x\in\mathbb R^2: |\boldsymbol x|<\hat R\}$ be the
balls with radii $R$ and $\hat R$, respectively, where $R>\hat R>0$. Denote by
$\partial B_R$ and $\partial B_{\hat R}$ the boundaries of $B_R$ and $R_{\hat
R}$, respectively. Let $\hat R$ be large enough such that $\overline D\subset
B_{\hat R}\subset B_R$. Denote by $\Omega=B_R\setminus\overline D$ the bounded
domain where the boundary value problem will be formulated.

\begin{figure}
\centering
\includegraphics[width=0.35\textwidth]{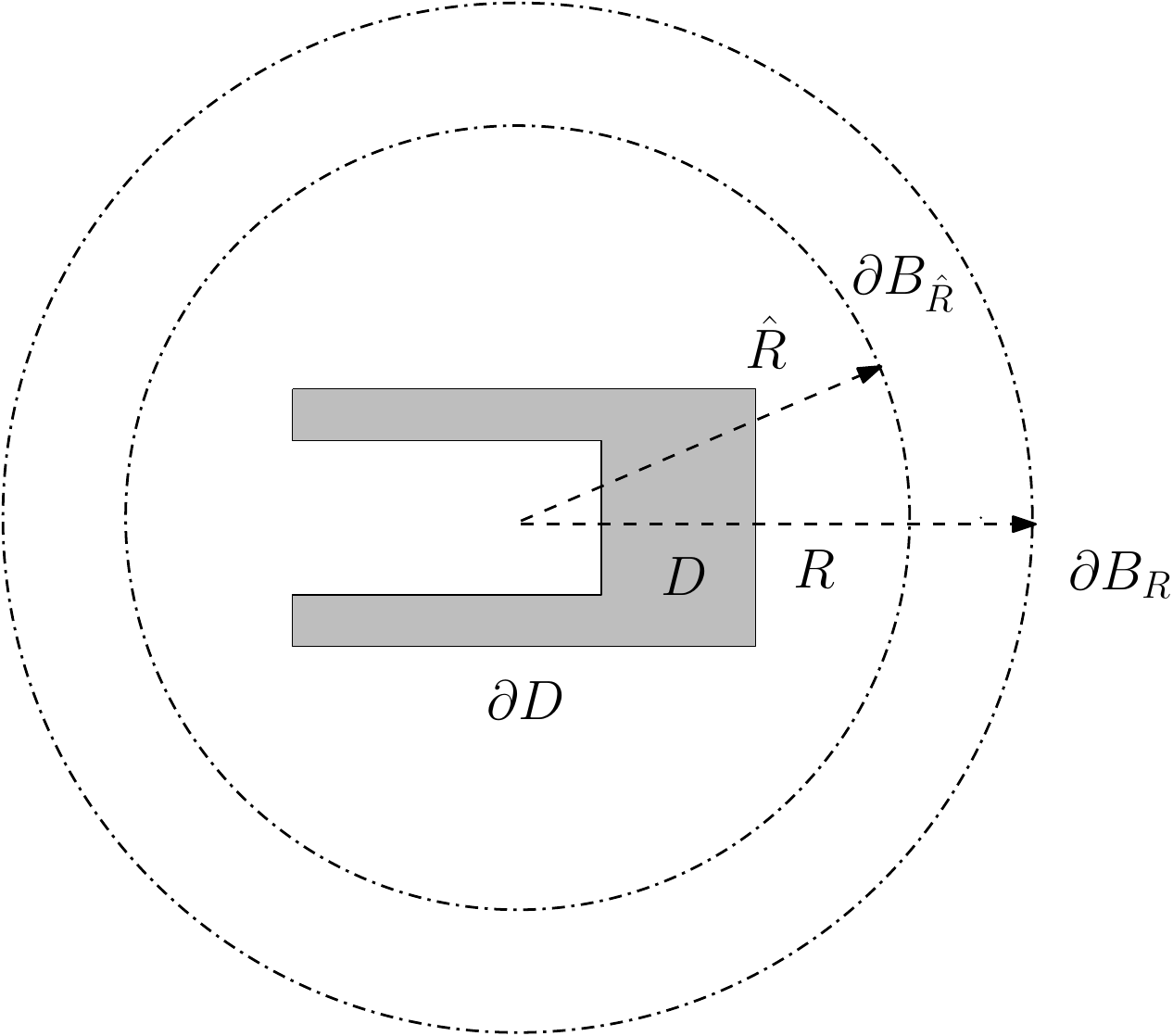}
\caption{Schematic of the elastic wave scattering problem.}
\label{fig1}
\end{figure}

Let the obstacle be illuminated by an incident wave $\boldsymbol u^{\rm inc}$.
The displacement of the scattered field $\boldsymbol u$ satisfies the
two-dimensional elastic wave equation
\begin{equation}\label{Navier-1}
\mu\Delta \boldsymbol{u}+(\lambda+\mu)\nabla\nabla\cdot \boldsymbol{u}
+\omega^2 \boldsymbol{u}=0\quad \text{in}\,\mathbb{R}^2\setminus\overline{D},
\end{equation}
where $\omega>0$ is the angular frequency and $\lambda, \mu$ are the Lam\'{e}
constants satisfying $\mu>0, \lambda+\mu>0$. Since the obstacle is assumed to
be rigid, the displacement of the total field $\boldsymbol
u+\boldsymbol u^{\rm inc}$ vanishes on the boundary of the obstacle, i.e., we
have
\begin{equation}\label{Navier-2}
 \boldsymbol u=\boldsymbol g\quad\text{on}\,\partial D,
\end{equation}
where $\boldsymbol g=-\boldsymbol u^{\rm inc}$. In addition, the scattered
field $\boldsymbol u$ is required to satisfy the Kupradze--Sommerfeld radiation
condition
\begin{equation}\label{Navier-3}
 \lim_{\rho\to\infty}\rho^{1/2}(\partial_\rho\boldsymbol u_{\rm p}-{\rm
i}\kappa_1 \boldsymbol u_{\rm p})=0,\quad
\lim_{\rho\to\infty}\rho^{1/2}(\partial_\rho\boldsymbol u_{\rm s}-{\rm
i}\kappa_2 \boldsymbol u_{\rm s})=0,\quad \rho=|\boldsymbol x|,
\end{equation}
where 
\[
 \boldsymbol u_{\rm p}=-\frac{1}{\kappa_1^2}\nabla\nabla\cdot\boldsymbol
u,\quad \boldsymbol u_{\rm s}=\frac{1}{\kappa_2^2}{\bf curl}{\rm
curl}\boldsymbol u,
\]
are the compressional and shear wave components of $\boldsymbol u$,
respectively. Here 
\[
 \kappa_1=\frac{\omega}{(\lambda+2\mu)^{1/2}},\quad
\kappa_2=\frac{\omega}{\mu^{1/2}}
\]
are knowns as the compressional wavenumber and the shear wavenumber,
respectively. Clearly we have $\kappa_1<\kappa_2$ since $\mu>0,
\lambda+\mu>0$. Given a vector function $\boldsymbol u=(u_1, u_2)^\top$ and a
scalar function $u$, the scalar and vector curl operators are defined by 
\[
 {\rm curl}\boldsymbol u=\partial_x u_2 -\partial_y u_1,\quad {\bf
curl}u=(\partial_y u, -\partial_x u)^\top.
\]

For any solution $\boldsymbol u$ of \eqref{Navier-1}, we introduce the
Helmholtz decomposition
\begin{equation}\label{hd}
 \boldsymbol u=\nabla\phi+{\bf curl}\psi,
\end{equation}
where $\phi, \psi$ are called the scalar potential functions. Substituting
\eqref{hd} into \eqref{Navier-1} yields that $\phi, \psi$ satisfy the Helmholtz
equation
\begin{equation}\label{he}
 \Delta \phi+\kappa_1^2\phi=0,\quad \Delta\psi+\kappa_2^2\psi=0\quad\text{in}\,
\mathbb R^2\setminus\overline D. 
\end{equation}
Taking the dot product of \eqref{Navier-2} with $\nu$ and $\tau$, respectively,
we get 
\begin{equation}\label{bc}
\partial_{\nu} \phi-\partial_{\tau}\psi=f_1,
\quad \partial_{\nu}\psi+\partial_{\tau}\phi=f_2  \quad\text{on}\,\partial D,
\end{equation}
where $f_1=-\boldsymbol g\cdot\nu$ and $f_2=\boldsymbol g\cdot\tau$. It follows
from \eqref{Navier-3} that $\phi, \psi$ satisfies the Sommerfeld radiation
condition
\begin{equation}\label{rc}
 \lim_{\rho\to\infty}\rho^{1/2}(\partial_\rho\phi-\kappa_1\phi)=0,\quad
\lim_{\rho\to\infty}\rho^{1/2}(\partial_\rho\psi-\kappa_2\psi)=0.
\end{equation}

Based on the Helmholtz decomposition, it is easy to show the equivalence
of the boundary value problems \eqref{Navier-1}--\eqref{Navier-3} and
\eqref{he}--\eqref{rc}. The details are omitted for brevity. 

\begin{lemma}
 Let $\boldsymbol u$ be the solution of the boundary value problem
\eqref{Navier-1}--\eqref{Navier-3}. Then
$\phi=-\kappa_1^{-2}\nabla\cdot\boldsymbol u, \psi=\kappa_2^{-1}{\rm
curl}\boldsymbol u$ are the solutions of the coupled boundary value problem
\eqref{he}--\eqref{rc}. Conversely, if $\phi, \psi$ are the solution of the
boundary value problem \eqref{he}--\eqref{rc}, then $\boldsymbol
u=\nabla\phi+{\bf curl}\psi$ is the solution of the boundary value problem
\eqref{Navier-1}--\eqref{Navier-3}.
\end{lemma}

Denote by $L^2(\Omega)$ the usual Hilbert space of square integrable
functions. Let $H^1(\Omega)$ be the standard Sobolev space equipped with the
norm
\[
 \| u\|_{H^1(\Omega)}=\Big(\|u\|^2_{L^2(\Omega)}+\|\nabla u\|^2_{L^2(\Omega)}
\Big)^{1/2}.
\]
Define $H^1_{\partial D}(\Omega)=\{u\in H^1(\Omega): u=0 \text{ on }\partial
D\}$. For any function $u\in L^2(\partial B_R)$, it admits the Fourier series
expansion
\[
 u(R, \theta)=\sum_{n\in\mathbb Z}\hat u_n(R) e^{{\rm i}n\theta},\quad \hat
u_n(R)=\frac{1}{2\pi}\int_0^{2\pi}u(R,\theta) e^{-{\rm i}n\theta}{\rm d}\theta.
\]
The trace space $H^s(\partial B_R), s\in\mathbb R$ is defined by 
\[
 H^s(\partial B_R)=\{u\in L^2(\partial B_R): \|u\|_{H^s(\partial B_R)}<\infty\},
\]
where $H^s(\partial B_R)$ norm is given by 
\[
 \|u\|_{H^s(\partial B_R)}=\Big(2\pi\sum_{n\in\mathbb Z}(1+n^2)^s|\hat u_n(R)|^2
\Big)^{1/2}.
\]
Let $\boldsymbol H^1(\Omega)=H^1(\Omega)^2$ and $\boldsymbol H^1_{\partial
D}(\Omega)=H^1_{\partial D}(\Omega)^2$ be the
Cartesian product spaces equipped with the corresponding 2-norms of
$H^1(\Omega)$ and $H^1_{\partial D}(\Omega)$, respectively. Throughout the
paper, we take the notation of $a\lesssim b$ to stand for $a\leq C b$, where
$C$ is a positive constant whose value is not required but should be clear from
the context. 

The elastic wave scattering problem \eqref{Navier-2}--\eqref{Navier-3} is
formulated in the open domain $\mathbb R^2\setminus\overline D$, which needs to
be truncated into the bounded domain $\Omega$. An appropriate boundary
condition is required on $\partial B_R$. 

Define a boundary operator for the displacement of the scattered wave
\[
 \mathscr B\boldsymbol u=\mu\partial_r \boldsymbol
u+(\lambda+\mu)\nabla\cdot\boldsymbol u \boldsymbol e_r\quad\text{on}\,\partial
B_R,
\]
where $\boldsymbol e_r$ is the unit outward normal vector on $\partial B_R$. It
is shown in \cite{LWWZ-ip16} that the scattered field $\boldsymbol u$ satisfies
the transparent boundary condition on $\partial B_R$:  
\begin{equation}\label{tbc}
\mathscr B\boldsymbol u=(\mathscr T\boldsymbol u)(R, \theta):=\sum_{n\in\mathbb
Z}M_n\boldsymbol u_n(R) e^{{\rm i}n\theta},\quad  \boldsymbol u(R,
\theta)=\sum_{n\in\mathbb Z}\boldsymbol u_n(R) e^{{\rm
i}n\theta}, 
\end{equation}
where $\mathscr T$ is called the Dirichlet-to-Neumann (DtN) operator and $M_n$
is a $2\times 2$ matrix whose entries are given in Appendix A. 

Based on the transparent boundary condition \eqref{tbc}, the
variational problem for \eqref{Navier-1}--\eqref{Navier-3} is to find
$\boldsymbol{u}\in \boldsymbol{H}^{1}(\Omega)$
with $\boldsymbol{u}=\boldsymbol{g}$ on $\partial D$ such that 
\begin{equation}\label{Variational_1}
b(\boldsymbol{u}, \boldsymbol{v})=0,\quad \forall \boldsymbol{v}\in
\boldsymbol{H}^{1}_{\partial D}(\Omega),
\end{equation}
where the sesquilinear form $b: \boldsymbol H^1(\Omega)\times\boldsymbol
H^1(\Omega)\to\mathbb C$ is defined as
\begin{align}\label{Bilinear_1}
b(\boldsymbol{u},
\boldsymbol{v})=\mu\int_{\Omega}\nabla\boldsymbol{u}:\nabla\overline{\boldsymbol
{v}}{\rm d}\boldsymbol{x}
+(\lambda+\mu)\int_{\Omega}\left(\nabla\cdot\boldsymbol{u}
\right)\left(\nabla\cdot\overline{\boldsymbol{v}}\right){\rm
d}\boldsymbol{x}\notag\\
-\omega^2 \int_{\Omega}  \boldsymbol{u}\cdot
\overline{\boldsymbol{v}} {\rm d}\boldsymbol{x}-
\int_{\partial
B_R}\mathscr{T}\boldsymbol{u}\cdot\overline{\boldsymbol{v}}{\rm d}s.
\end{align}
Here $A:B={\rm tr}(AB^\top)$ is the Frobenius inner product of square matrices
$A$ and $B$. 

Following \cite{LWWZ-ip16}, we may show that the variational
problem \eqref{Variational_1} has a unique weak solution $\boldsymbol u\in
\boldsymbol H^1(\Omega)$ for any frequency $\omega$ and the solution satisfies
the estimate 
\begin{equation}\label{vp_se}
 \|\boldsymbol u\|_{\boldsymbol H^1(\Omega)}\lesssim \|\boldsymbol
g\|_{\boldsymbol H^{1/2}(\partial D)}\lesssim \|\boldsymbol u^{\rm
inc}\|_{\boldsymbol H^1(\Omega)}. 
\end{equation}
It follows from the general theory in \cite{BA-73} that there exists a constant
$\gamma>0$ such that the following inf-sup condition holds
\[
 \sup_{0\neq\boldsymbol v\in\boldsymbol H^1(\Omega)}\frac{|b(\boldsymbol u,
\boldsymbol v)|}{\|\boldsymbol v\|_{\boldsymbol H^1(\Omega)}}\geq\gamma
\|\boldsymbol u\|_{\boldsymbol H^1(\Omega)},\quad\forall \boldsymbol
u\in\boldsymbol H^1(\Omega).
\]

\section{The discrete problem}\label{Section_dp}

Let us consider the discrete problem of \eqref{Variational_1} by using the
finite element approximation. Let $\mathcal M_h$ be a regular triangulation of
$\Omega$, where $h$ denotes the maximum diameter of all the elements in
$\mathcal M_h$. For simplicity, we assume that the boundary $\partial D$ is
polygonal and ignore the approximation error of the boundary $\partial B_R$,
which allows to focus of deducing the a posteriori error estimate. Thus any edge
$e\in\mathcal M_h$ is a subset of $\partial \Omega$ if it has two boundary
vertices.

Let $\boldsymbol V_h\subset \boldsymbol{H}^{1}(\Omega)$ be a conforming finite
element space, i.e., 
\[
\boldsymbol V_h:= \left\{\boldsymbol v\in  C(\overline\Omega)^2:
\boldsymbol v\big|_{K}\in P_m(K)^2 \text{ for any } K\in\mathcal
M_h\right\},
\]
where $m$ is a positive integer and $P_m(K)$ denotes the set of all
polynomials of degree no more than $m$. The finite element approximation to the
variational problem \eqref{Variational_1} is to find $\boldsymbol{u}^h\in
\boldsymbol V_h$ with $\boldsymbol u^h=\boldsymbol g$ on $\partial D$ such that
\begin{equation}\label{Variational_3}
b(\boldsymbol{u}^h, \boldsymbol{v}^h)=0,\quad \forall
\boldsymbol{v}^h\in \boldsymbol V_{h,\partial D},
\end{equation}
where $\boldsymbol V_{h,\partial D}=\{\boldsymbol v\in \boldsymbol V_h:
\boldsymbol v=0\text{ on }\partial D\}$.

In the variational problem \eqref{Variational_3}, the DtN operator $\mathscr T$
is given by an infinite series. In practical computation, the infinite series
must be truncated into a finite sum. Given a sufficiently large $N$, we define
the truncated DtN operator
\begin{equation}\label{Truncated_TBC}
\mathscr{T}_N \boldsymbol{u}=\sum\limits_{|n|\leq N}M_n\boldsymbol u_n(R)
e^{{\rm i}n\theta}.
\end{equation} 
Using \eqref{Truncated_TBC}, we have the truncated finite element
approximation: Find $\boldsymbol u_N^h\in\boldsymbol V_h$ with $\boldsymbol
u_N^h=\boldsymbol g$ on $\partial D$ such that
\begin{equation}\label{Variational_2}
b_N(\boldsymbol{u}^h_N, \boldsymbol{v}^h)=0,\quad \forall \boldsymbol{v}^h\in
\boldsymbol V_{h, \partial D},
\end{equation}
where the sesquilinear form $b_N:\boldsymbol V_h\times\boldsymbol V_h\to\mathbb
C$ is defined as
\begin{align}\label{Bilinear_2}
b_N(\boldsymbol{u},
\boldsymbol{v})=\mu\int_{\Omega}\nabla\boldsymbol{u}:\nabla\overline{
\boldsymbol{v}}{\rm d}\boldsymbol{x}
+(\lambda+\mu)\int_{\Omega}\left(\nabla\cdot\boldsymbol{u}
\right)\left(\nabla\cdot\overline{\boldsymbol{v}}\right){\rm d}\boldsymbol{x}
\notag\\
-\omega^2 \int_{\Omega} 
\boldsymbol{u}\cdot \overline{\boldsymbol{v}}{\rm d}\boldsymbol{x}
-\int_{\partial
B_R}\mathscr{T}_N\boldsymbol{u}\cdot\overline{\boldsymbol{v}}{\rm d}s.
\end{align}

For sufficiently large $N$ and sufficiently small $h$, the discrete inf-sup
condition of the sesquilinear form $b_N$ may be established by following the
approach in \cite{S-mc74}. Based on the general theory in \cite{BA-73}, the
truncated variational problem \eqref{Variational_2} can be shown to have a
unique solution $\boldsymbol u_N^h\in\boldsymbol V_h$. The details are omitted
since our focus is the a posteriori error estimate.

\section{The a posterior error analysis}\label{Section_pea}

For any triangular element $K\in\mathcal M_h$, denoted by $h_K$ its diameter.
Let $\mathcal B_h$ denote the set of all the edges of $K$. For any edge
$e\in\mathcal B_h$, denoted by $h_e$ its length. For any interior edge $e$ which
is the common side of triangular elements $K_1, K_2\in \mathcal M_h$, we
define the jump residual across $e$ as
\[
J_e=-\left[\mu\nabla \boldsymbol{u}_N^h\cdot\boldsymbol{\nu}
_1+(\lambda+\mu)\nabla(\nabla\cdot\boldsymbol{u}_N^h)
\cdot\boldsymbol{\nu}_1+\mu\nabla
\boldsymbol{u}_N^h\cdot \boldsymbol{\nu}_2+(\lambda+\mu)		
\nabla(\nabla\cdot\boldsymbol{u}_N^h)\cdot\boldsymbol{\nu}_2\right],
\]
where $\boldsymbol{\nu}_j$ is the unit outward normal vector on the boundary of
$K_j, j=1,2$. For any boundary edge $e\subset \partial B_R$, we define the jump
residual
\[	
J_e=2\left(\mathscr{T}_N\boldsymbol{u}-\mu(\nabla\boldsymbol{u}
_N^h\cdot\boldsymbol{e}_r)-(\lambda+\mu)\nabla(\nabla\cdot\boldsymbol
{u} _N^h)\boldsymbol{e}_r\right).
\]
For any triangular element $K\in\mathcal M_h$, denote by $\eta_K$ the local
error estimator which is given by
\[
\eta_K=h_K\|\mathscr R\boldsymbol{u}_N^h\|_{\boldsymbol{L}^2(K)}
+\left(\frac{1}{2}\sum\limits_{e\in \partial K} h_e
\|J_e\|^2_{\boldsymbol{L}^2(e)}\right)^{1/2},
\]
where $\mathscr R$ is the residual operator defined by 
\[
\mathscr R\boldsymbol{u}=\mu\Delta\boldsymbol{u}
+(\lambda+\mu)\nabla\left(\nabla\cdot\boldsymbol{u}\right)+\omega^2\boldsymbol{u
}.
\]

For convenience, we introduce a weighted norm
$\vvvert\cdot\vvvert_{\boldsymbol{H}^{1}(\Omega)}$ which is given by 
\begin{equation}\label{eqn}
\vvvert\boldsymbol{u}\vvvert^2_{\boldsymbol{H}^{1}(\Omega)}
=\mu\int_{\Omega} 
|\nabla\boldsymbol{u}|^2{\rm d}\boldsymbol{x}
+(\lambda+\mu)\int_{\Omega}|\nabla\cdot\boldsymbol{u}|^2{\rm d}\boldsymbol{x}
+\omega^2 \int_{\Omega} |\boldsymbol{u}|^2{\rm d}\boldsymbol{x}.
\end{equation} 
It can be verified for any $\boldsymbol u\in \boldsymbol H^1(\Omega)$ that
\begin{equation}\label{Equal_norm}
\min \left(\mu, \omega^2\right)
\|\boldsymbol{u}\|^2_{\boldsymbol H^{1}(\Omega)}
\leq \vvvert\boldsymbol{u}\vvvert^2_{\boldsymbol H^1(\Omega)}
\leq \max\left(2\lambda+3\mu,
\omega^2\right)\|\boldsymbol{u}\|^2_{\boldsymbol H^{1}(\Omega)},
\end{equation}
which implies that the two norms $\|\cdot\|_{\boldsymbol H^1(\Omega)}$ and
$\vvvert\cdot\vvvert_{\boldsymbol H^1(\Omega)}$ are equivalent.

Now we state the main result of this paper. 

\begin{theorem}\label{Main_Result}
Let $\boldsymbol{u}$ and $\boldsymbol{u}_N^h$ be the solution of the variational
problem \eqref{Variational_1} and \eqref{Variational_3}, respectively. Then for
sufficiently large $N$, the following a posterior error estimate holds 
\begin{eqnarray*}
\vvvert\boldsymbol{u}-\boldsymbol{u}_N^h
\vvvert^2_{\boldsymbol{H}^{1}(\Omega)}\lesssim 
\left(\sum_{K\in\mathcal M_h} \eta^2_{K}\right)^{1/2}
+\max_{|n|>N}\left(|n|\bigg(\frac{\hat{R}}{R}\bigg)^{|n|}\right)
\|\boldsymbol{u}^{\rm inc}\|_{\boldsymbol{H}^{1}(\Omega)}.
\end{eqnarray*}
\end{theorem}

We point out that the a posteriori error estimate consists of two parts: the
first part arises from the finite element discretization error; the second part
comes from the truncation error of the DtN operator. Apparently, the DtN
truncation error decreases exponentially with respect to $N$ since $\hat{R}<R$. 

In the rest of the paper, we shall prove the a posteriori error estimate in Theorem
\ref{Main_Result}. First, we present the result on trace regularity for functions in $H^1(\Omega)$. The proof can be found in \cite{JLLZ-jsc17}. 

\begin{lemma}\label{Poincare}
For any $u\in H^{1}(\Omega)$, the following estimates hold
\[
\|u\|_{H^{1/2}(\partial B_R)}\lesssim \|u\|_{H^{1}(\Omega)},
\quad \|u\|_{H^{1/2}(\partial B_{\hat{R}})}\lesssim
\|u\|_{H^{1}(\Omega)}.
\]
\end{lemma}

Let $\boldsymbol{\xi}=\boldsymbol{u}-\boldsymbol{u}_N^h$, where $\boldsymbol u$
and $\boldsymbol{u}_N^h$ are the solutions of the problems \eqref{Variational_1}
and \eqref{Variational_3}, respectively. Combining \eqref{eqn},
\eqref{Bilinear_1}, and \eqref{Bilinear_2}, we have from
straightforward calculations that 
\begin{eqnarray}\label{LemmaXi_1}
\vvvert\boldsymbol{\xi}\vvvert^2_{\boldsymbol{H}^{1}(\Omega)}
&=& \mu\int_{\Omega} \nabla\boldsymbol{\xi}:\nabla\overline{\boldsymbol{\xi}}
{\rm d}\boldsymbol{x}
+(\lambda+\mu)\int_{\Omega}\left(\nabla\cdot\boldsymbol{\xi}
\right)\left(\nabla\cdot\overline{\boldsymbol{\xi}}\right){\rm d}\boldsymbol{x}
+\omega^2 \int_{\Omega}
\boldsymbol{\xi}\cdot\overline{\boldsymbol{\xi}}{\rm d}\boldsymbol{x}\notag\\
&=& \Re b(\boldsymbol{\xi}, \boldsymbol{\xi})+2\omega^2
\int_{\Omega}|\boldsymbol{\xi}|^2{\rm
d}\boldsymbol{x} +\Re\int_{\partial
B_R}\mathscr{T}\boldsymbol{\xi}\cdot \overline{\boldsymbol{\xi}}{\rm d}s\notag\\
&=& \Re b(\boldsymbol{\xi}, \boldsymbol{\xi})+\Re\int_{\partial
B_R}\left(\mathscr{T}
-\mathscr{T}_N\right)\boldsymbol{\xi}\cdot\overline{\boldsymbol{\xi}}{\rm d}s +2\omega^2
\int_{\Omega}|\boldsymbol{\xi}|^2{\rm
d}\boldsymbol{x}
+\Re\int_{\partial B_R}\mathscr{T}_N
\boldsymbol{\xi}\cdot\overline{\boldsymbol{\xi}}{\rm d}s,
\end{eqnarray}
which is the error representation formula. 

In the following, we discuss the four terms in \eqref{LemmaXi_1} one by one. 
Lemma \ref{lemma_xi_5} presents the a posteriori error estimate for the
truncation of the DtN operator. Lemma \ref{lemma_xi_3}
gives the a posteriori error estimate for both of the finite element approximation
and DtN operator truncation. 

\begin{lemma}\label{lemmaHankel}
Let $0<\kappa_1<\kappa_2$ and $0<\hat R<R$. For sufficiently large $|n|$, the
following estimate holds for $j=1, 2$:
\begin{equation*}
\left|\frac{H_{n}^{(j)}(\kappa_1 R)}{H_{n}^{(j)}(\kappa_1
\hat R)}-\frac{H_{n}^{(j)}(\kappa_2 R)}{H_{n}^{(j)}(\kappa_2 \hat R)}\right|
\leq \frac{\kappa_2\left(\kappa_2-\kappa_1\right)}{|n|-1}
\left(R^2-\hat R^2\right)\bigg(\frac{\hat R}{R}\bigg)^{|n|},
\end{equation*}
where $H_n^{(1)}$ and $H_n^{(2)}$ are the Hankel functions of the first and
second kind with order $n$, respectively.  
\end{lemma} 

\begin{proof}
Since the Hankel functions of the first and second kind are complex conjugate
to each other, we only need to show the proof for the Hankel function of the
first kind. 

Let $J_n$ and $Y_n$ be the Bessel functions of the first and second kind with
order $n$, respectively. Since $J_{-n}(z)=(-1)^n J_n(z),
Y_{-n}(z)=(-1)^nY_n(z)$, it suffices to show the result for positive $n$. For a
fixed $z>0$, they admit the following asymptotic
properties \cite{W-95}:
\begin{equation}\label{lemmaHankel-s1}
J_n(z)\sim \frac{1}{\sqrt{2\pi
n}}\left(\frac{ez}{2n}\right)^n,\qquad
Y_n(z)\sim -\sqrt{\frac{2}{\pi
n}}\left(\frac{ez}{2n}\right)^{-n},\quad n\to\infty. 
\end{equation} 
Define $S(z)=J_n(z)/Y_n(z)$. A simple calculation yields 
\begin{eqnarray}\label{lemmaHankel-s2}
\frac{H_n^{(1)}(zR)}{H_n^{(1)}(z\hat R)} &=&
\frac{J_n(zR)+{\rm i}Y_n(zR)}{J_n(z\hat{R})+{\rm i}Y_n(z\hat R)}
=\frac{Y_n(zR)}{Y_n(z\hat R)}\frac{1-{\rm
i}\frac{J_n(zR)}{Y_n(zR)}}{1-{\rm i}\frac{J_n(z\hat R)}{Y_n(z\hat R)}}\notag\\ 
  &=&\frac{Y_n(zR)}{Y_n(z\hat R)}\frac{1-{\rm i}S_n(zR)}{1-{\rm i}S_n(z\hat
R)}=\frac{Y_n(zR)}{Y_n(z\hat R)}+{\rm
i}\frac{Y_n(zR)}{Y_n(z\hat R)}\frac{S_n(z\hat R)-S_n(zR)}{1-{\rm
i}S_n(z\hat R)}.
\end{eqnarray}
By \eqref{lemmaHankel-s1}--\eqref{lemmaHankel-s2}, we have
\begin{eqnarray*}
S_n(z)=\frac{J_n(z)}{Y_n(z)}\sim
\frac{\frac{1}{\sqrt{2\pi n}}
\left(\frac{ez}{2n}\right)^n}{-\sqrt{\frac{2}{\pi
n}}\left(\frac{ez}{2n}\right)^{-n}}
\sim
-\frac{1}{2}\left(\frac{ez}{2n}\right)^{2n}
\end{eqnarray*}
and
\begin{eqnarray*}
&& \left|\frac{H_{n}^{(1)}(\kappa_1
R)}{H_{n}^{(1)}(\kappa_1 \hat R)}-\frac{H_{n}^{(1)}(\kappa_2
R)}{H_{n}^{(1)}(\kappa_2 \hat R)}\right|
\leq \left|\frac{Y_n(\kappa_1 R)}{Y_n(\kappa_1
\hat R)}-\frac{Y_n(\kappa_2 R)}{Y_n(\kappa_2 \hat R)}\right|
+\left|\frac{Y_n(\kappa_1 R)}{Y_n(\kappa_1
\hat R)}\frac{S_n(\kappa_1 \hat R)}{1-{\rm i}S_n(\kappa_1 \hat R)}\right| \\
&&\quad +\left|\frac{Y_n(\kappa_1
R)}{Y_n(\kappa_1 \hat R)}\frac{S_n(\kappa_1 R)}{1-{\rm i}S_n(\kappa_1
\hat R)}\right| +\left|\frac{Y_n(\kappa_2 R)}{Y_n(\kappa_2
\hat R)}\frac{S_n(\kappa_2 \hat R)}{1-{\rm i}S_n(\kappa_2 \hat R)}\right|
+\left|\frac{Y_n(\kappa_2 R)}{Y_n(\kappa_2
\hat R)}\frac{S_n(\kappa_2 R)}{1-{\rm i}S_n(\kappa_2 \hat R)}\right|.
\end{eqnarray*}
It is easy to verify that 
\begin{eqnarray*}
\left|\frac{S_n(z R)}{1-{\rm i}S_n(z
\hat R)}\right|\leq \left(\frac{ezR}{2n}\right)^{2n},\quad
\left|\frac{S_n(z \hat R)}{1-{\rm i}S_n(z
\hat R)}\right|\leq \left(\frac{ez\hat R}{2n}\right)^{2n}
\end{eqnarray*}
and
\begin{eqnarray*}
\frac{Rz Y^\prime_n(z R)}{Y_n(z R)} \sim \frac{z^2 R^2}{2(n-1)}-n,\quad
\frac{Y_n(z R)}{Y_n(z \hat R)}\sim \left(\frac{\hat R}{R}\right)^{|n|}.
\end{eqnarray*}
Combining the above estimates, we have for $R>\hat R$ and $\kappa_2>\kappa_1$
that 
\begin{eqnarray*}
&& \left|\frac{Y_n(\kappa_1 R)}{Y_n(\kappa_1
\hat R)}\frac{S_n(\kappa_1 \hat R)}{1-{\rm i}S_n(\kappa_1 \hat R)}\right| 
+\left|\frac{Y_n(\kappa_1 R)}{Y_n(\kappa_1
\hat R)}\frac{S_n(\kappa_1 R)}{1-{\rm i}S_n(\kappa_1 \hat R)}\right| 
+\left|\frac{Y_n(\kappa_2 R)}{Y_n(\kappa_2
\hat R)}\frac{S_n(\kappa_2 \hat R)}{1-{\rm i}S_n(\kappa_2 \hat R)}\right|\\
&&\quad +\left|\frac{Y_n(\kappa_2
R)}{Y_n(\kappa_2 \hat R)}\frac{S_n(\kappa_2 R)}{1-{\rm i}S_n(\kappa_2
\hat R)}\right| \leq 2\left(\frac{e\kappa_2
R}{2n}\right)^{2n}\left(\left|\frac{Y_n(\kappa_1 R)}{Y_n(\kappa_1
\hat R)}\right| +\left|\frac{Y_n(\kappa_2
R)}{Y_n(\kappa_2 \hat R)}\right|\right).
\end{eqnarray*}

Define $F(z)=Y_n(z R)/Y_n(z \hat R)$. By the mean value
theorem, there exits $\xi\in(\kappa_1, \kappa_2)$ such that
\begin{eqnarray*}
F(\kappa_1)-F(\kappa_2) &=&
F'(\xi)(\kappa_1-\kappa_2)\\
&=& \frac{R Y'_n(\xi R) Y_n(\xi \hat R)- \hat R 
Y_n(\xi R) Y'_n(\xi \hat R)}{Y_n(\xi \hat R)^2}(\kappa_1-\kappa_2)\\
&=& \left(\frac{R\xi Y'_n(\xi R)}{Y_n(\xi
R)}-\frac{\hat R\xi Y'_n(\xi \hat R)}{Y_n(\xi \hat R)}\right)
\frac{Y_n(\xi R)}{Y_n(\xi \hat R)}\frac{\kappa_1-\kappa_2}{\xi}\\
&\sim& \frac{\xi\left(\kappa_1-\kappa_2\right)}{2(n-1)}
\left(R^2-\hat R^2\right)\frac{ Y_n(\xi R)}{Y_n(\xi \hat R)}.
\end{eqnarray*}
Therefore,
\begin{eqnarray*}
\left|\frac{H_{n}^{(1)}(\kappa_1
R)}{H_{n}^{(1)}(\kappa_1 \hat R)}-\frac{H_{n}^{(1)}(\kappa_2
R)}{H_{n}^{(1)}(\kappa_2 \hat R)}\right|
&\leq& \left|\frac{\xi\left(\kappa_1-\kappa_2\right)}{2(n-1)}
\right|\left|\left(R^2-\hat R^2\right)\frac{Y_n(\xi R)}{Y_n(\xi
\hat R)}\right|\\
&&\quad +2\left(\frac{e\kappa_2
R}{2n}\right)^{2n}\left(\left|\frac{Y_n(\kappa_1 R)}{Y_n(\kappa_1
\hat R)}\right| +\left|\frac{Y_n(\kappa_2
R)}{Y_n(\kappa_2 \hat R)}\right|\right)\\
&\leq &\frac{\kappa_2\left(\kappa_2-\kappa_1\right)}{|n|-1}
\left(R^2-\hat R^2\right)\left(\frac{\hat R}{R}\right)^{|n|},
\end{eqnarray*}
which completes the proof. 
\end{proof}

\begin{lemma}\label{lemma_xi_5}
Let $\boldsymbol u\in\boldsymbol H^1(\Omega)$ be the solution of the
variational problem \eqref{Variational_1}. For any $\boldsymbol v\in\boldsymbol
H^1(\Omega)$, the following estimate holds
\begin{equation*}
\left|\int_{\partial
B_R}\left(\mathscr{T}-\mathscr{T}_N\right)\boldsymbol{u}\cdot
\overline{\boldsymbol{v}}\,ds \right|
\leq C\max_{|n|>N}\left(|n|\bigg(\frac{\hat R}{R}\bigg)^{|n|}\right)
\|\boldsymbol{u}^{\rm inc}\|_{\boldsymbol{H}^{1}(\Omega)}
\|\boldsymbol{v}\|_{\boldsymbol{H}^{1}(\Omega)},
\end{equation*}
where $C$ is a positive constant independent of $N$.
\end{lemma}

\begin{proof}
Recalling the Helmholtz decomposition $\boldsymbol u=\nabla\phi+{\bf
curl}\psi$, we have from the Fourier series expansions in \eqref{Fourier_Expan}
that 
\begin{equation*}
\phi_n(R)=\frac{H_n^{(1)}(\kappa_1 R)}{H_{n}^{(1)}(\kappa_1
\hat R)}\phi_n(\hat R), \quad \psi_n(R)=\frac{H_n^{(1)}(\kappa_2
R)}{H_{n}^{(1)}(\kappa_2 \hat R)}\psi_n(\hat R).
\end{equation*}
Comparing the Fourier coefficients of $\boldsymbol u$ and $\phi, \psi$ in the
Helmholtz decomposition gives
\begin{eqnarray*}
\boldsymbol{u}_n(R) &=&
\begin{bmatrix}
\alpha_{1n}(R) & \frac{{\rm i}n}{R}\\[2pt]
\frac{{\rm i}n}{R} & -\alpha_{2n}(R)
\end{bmatrix}
\begin{bmatrix}
\phi_n(R)\\ \psi_n(R)
\end{bmatrix}\\
&=&\begin{bmatrix}
\alpha_{1n}(R) & \frac{{\rm i}n}{R}\\[2pt]
\frac{{\rm i}n}{R} & -\alpha_{2n}(R)
\end{bmatrix}
\begin{bmatrix}
\frac{H_n^{(1)}(\kappa_1 R)}{H_{n}^{(1)}(\kappa_1 \hat R)}
& 0\\
0 & \frac{H_n^{(1)}(\kappa_2 R)}{H_{n}^{(1)}(\kappa_2
\hat R)}
\end{bmatrix}
\begin{bmatrix}
\phi_n(\hat R)\\ \psi_n(\hat R)
\end{bmatrix}\\
&=& \begin{bmatrix}
\alpha_{1n}(R) & \frac{{\rm i}n}{R}\\[2pt]
\frac{{\rm i}n}{R} & -\alpha_{2n}(R)
\end{bmatrix}
\begin{bmatrix}
\frac{H_n^{(1)}(\kappa_1 R)}{H_{n}^{(1)}(\kappa_1 \hat R)}
& 0\\
0 & \frac{H_n^{(1)}(\kappa_2 R)}{H_{n}^{(1)}(\kappa_2
\hat R)}
\end{bmatrix}
\begin{bmatrix}
-\alpha_{2n}(\hat R) & -\frac{{\rm i}n}{\hat R}\\[2pt]
-\frac{{\rm i}n}{\hat R} & \alpha_{1n}(\hat R)
\end{bmatrix}
\frac{\boldsymbol{u}_n(\hat R)}{\Lambda_n(\hat R)}\\
&=& -\frac{1}{\Lambda_n(\hat R)}\begin{bmatrix}
A_{11} & A_{12}\\ A_{21} & A_{22}
\end{bmatrix}\boldsymbol{u}_n(\hat R),
\end{eqnarray*}	
where $\alpha_{jn}, \Lambda_n$ is given in \eqref{alpha_jn} and 
\begin{eqnarray*}
A_{11} &=& \frac{H_n^{(1)}(\kappa_1 R)}{H_n^{(1)}(\kappa_1
\hat R)}\alpha_{1n}(R)\alpha_{2n}(\hat R)
-\frac{n^2}{R \hat R}\frac{H_n^{(1)}(\kappa_2
R)}{H_n^{(1)}(\kappa_2 \hat R)},\\
A_{12} &=& \frac{H_n^{(1)}(\kappa_1 R)}{H_n^{(1)}(\kappa_1
\hat R)}\alpha_{1n}(R)\frac{{\rm i}n}{\hat R}
-\frac{{\rm i}n}{R}\alpha_{1n}(\hat R)
\frac{H_n^{(1)}(\kappa_2 R)}{H_n^{(1)}(\kappa_2 \hat R)},\\
A_{21} &=& \frac{H_n^{(1)}(\kappa_1 R)}{H_n^{(1)}(\kappa_1
\hat R)}\alpha_{2n}(\hat R)\frac{{\rm i}n}{R}
-\frac{{\rm i}n}{\hat R}\alpha_{2n}(R)
\frac{H_n^{(1)}(\kappa_2 R)}{H_n^{(1)}(\kappa_2 \hat R)},\\	
A_{22} &=& \frac{H_n^{(1)}(\kappa_2 R)}{H_n^{(1)}(\kappa_2
\hat R)}\alpha_{1n}(\hat R)\alpha_{2n}(R)
-\frac{n^2}{R \hat R}\frac{H_n^{(1)}(\kappa_1
R)}{H_n^{(1)}(\kappa_1 \hat R)}.				
\end{eqnarray*}

By Lemma \ref{lemmaHankel}, we have 
\begin{eqnarray*}
\left|A_{11}\right| &\leq& \left| \frac{H_n^{(1)}(\kappa_1
R)}{H_n^{(1)}(\kappa_1 \hat R)}\right|
\left|\alpha_{1n}(R)\alpha_{2n}(\hat R)-\frac{n^2}{R
\hat R}\right| +\frac{n^2}{R \hat R}\left|\frac{H_n^{(1)}(\kappa_1
R)}{H_n^{(1)}(\kappa_1 \hat R)}-\frac{H_n^{(1)}(\kappa_2 R)}{H_n^{(1)}(\kappa_2
\hat R)}\right|\\
&\leq& \bigg(\frac{\hat R}{R}\bigg)^{|n|}\left|\left(\frac{\kappa_1^2
R}{2(n-1)}-\frac{n}{R}\right) \left(\frac{\kappa_2^2
\hat R}{2(n-1)}-\frac{n}{\hat R}\right)-\frac{n^2}{R \hat R}\right|\\
&&\quad +\frac{n^2}{R
\hat R}\frac{\kappa_2\left(\kappa_1-\kappa_2\right)}{|n|-1}
\left(R^2-\hat R^2\right)\bigg(\frac{\hat{R}}{R}\bigg)^{|n|}\\
&\leq& \kappa_2\left(\kappa_2-\kappa_1\right)\left(R^2-\hat
R^2\right)\frac{2|n|}{R \hat R}
\bigg(\frac{\hat R}{R}\bigg)^{|n|}\leq C|n|
\bigg(\frac{\hat R}{R}\bigg)^{|n|},
\end{eqnarray*}
where $C$ is a positive constant independent of $n$. Similarly, it can be shown
that there exists a positive constant $C$ independent of $n$ such that 
\[
 |A_{ij}|\leq C|n| \bigg(\frac{\hat R}{R}\bigg)^{|n|},\quad i,j=1, 2. 
\]
The proofs are omitted for brevity. 

Combining the above estimates and Lemma \ref{Lamn}, we obtain 
\[
\left|\boldsymbol{u}^{(n)}(R)\right|\leq
C|n|\bigg(\frac{\hat R}{R}\bigg)^{|n|}\left|\boldsymbol{u}^{(n)}(\hat
R)\right|. 
\]
Combining the above estimate with Lemma \ref{Poincare} and \eqref{vp_se}
yields 
\begin{eqnarray*}
&& \left|\int_{\partial
B_R}\left(\mathscr{T}-\mathscr{T}_N\right)\boldsymbol{u}\cdot
\overline{\boldsymbol{v}}{\rm d}s \right|
 =2\pi R\sum\limits_{|n|>N}\begin{bmatrix}
-\frac{\mu}{R}+\frac{\omega^2}{\Lambda_n}\alpha_{2n}(R) & 
{\rm i}n\left(-\frac{\mu}{R}+\frac{\omega^2}{\Lambda}\frac{1}{R}\right)\\
-{\rm i}n\left(-\frac{\mu}{R}+\frac{\omega^2}{\Lambda}\frac{1}{R}\right) & 
-\frac{\mu}{R}+\frac{\omega^2}{\Lambda}\alpha_{1n}(R)
\end{bmatrix}
\boldsymbol{u}_n(R) \cdot \overline{\boldsymbol{v}_n(R)}
\notag\\
&&\qquad \lesssim 2\pi R\sum\limits_{|n|>N}
|n|^2\bigg(\frac{\hat R}{R}\bigg)^{|n|} 
\boldsymbol{u}_n(\hat R)\cdot\overline{\boldsymbol{v}_n(R)} \\
&& \qquad \lesssim 2\pi R\sum\limits_{|n|>N} 
\bigg(\frac{\hat R}{R}\bigg)^{|n|} |n|
\left(|n|^{1/2}\boldsymbol{u}_n(\hat R)\right)\left(
|n|^{1/2}\overline{\boldsymbol{v}_n(R)}\right)\\
&& \qquad \leq 
C\max_{|n|>N}\left(|n|\bigg(\frac{\hat R}{R}\bigg)^{|n|}\right) 
\|\boldsymbol{u}\|_{\boldsymbol{H}^{1/2}(\partial
B_{\hat R})}\|\boldsymbol{v}\|_{ \boldsymbol{H}^{1/2}(\partial B_{R})}\\
&& \qquad  \leq
C\max_{|n|>N}\left(|n|\bigg(\frac{\hat R}{R}\bigg)^{|n|}\right)
\|\boldsymbol{u}\|_{\boldsymbol{H}^{1}(\Omega)}
\|\boldsymbol{v}\|_{\boldsymbol{H}^{1}(\Omega)}\\
&& \qquad \leq C\max_{|n|>N}\left(|n|\bigg(\frac{\hat R}{R}\bigg)^{|n|}\right)
\|\boldsymbol{u}^{\rm inc}\|_{\boldsymbol{H}^{1}(\Omega)}
\|\boldsymbol{v}\|_{\boldsymbol{H}^{1}(\Omega)},
\end{eqnarray*}
which completes the proof.
\end{proof}

In Lemma \ref{lemma_xi_5}, it is shown that the truncation error of the DtN
operator decay exponentially with respect to the truncation parameter $N$. The
result implies that $N$ can be small in practice. 

\begin{lemma}\label{lemma_xi_3}
Let $\boldsymbol{v}$ be any function in $\boldsymbol{H}^{1}(\Omega)$, we have
\begin{equation*}
\left|b(\boldsymbol{\xi}, \boldsymbol{v})+\int_{\partial B_R}
\left(\mathscr{T}-\mathscr{T}_N\right)\boldsymbol{\xi}\cdot 
\overline{\boldsymbol{v}}{\rm d}s\right|\lesssim
\left( \left(\sum\limits_{K\in\mathcal M_h}\eta_{K}^2\right)^{1/2}
+\max_{|n|>N}\left(|n|\bigg(\frac{\hat R}{R}\bigg)^{|n|}\right)
\|\boldsymbol{u}^{\rm inc}\|_{\boldsymbol{H}^{1}(\Omega)}\right)
\|\boldsymbol{v}\|_{\boldsymbol{H}^{1}(\Omega)}.
\end{equation*}
\end{lemma}

\begin{proof}
For any function $\boldsymbol{v}$ in $\boldsymbol{H}^{1}(\Omega)$, we  have
\begin{eqnarray*}
&& b(\boldsymbol{\xi}, \boldsymbol{v})
+\int_{\partial
B_R}\left(\mathscr{T}-\mathscr{T}_N\right)\boldsymbol{\xi}\cdot
\overline{\boldsymbol{v}}{\rm d}s  =b(\boldsymbol{u},
\boldsymbol{v})-b(\boldsymbol{u}_N^{h}, \boldsymbol{v})
+\int_{\partial B_R}\left(\mathscr{T}-\mathscr{T}_N\right)\boldsymbol{\xi}\cdot
\overline{\boldsymbol{v}}{\rm d}s  \\
&&\qquad =b(\boldsymbol{u},
\boldsymbol{v})-b_N(\boldsymbol{u}_N^h,
\boldsymbol{v})+b_N(\boldsymbol{u}_N^h, \boldsymbol{v})
-b(\boldsymbol{u}_N^{h}, \boldsymbol{v})
+\int_{\partial B_R}\left(\mathscr{T}-\mathscr{T}_N\right)\boldsymbol{\xi}\cdot
\overline{\boldsymbol{v}}{\rm d}s  \\  \\
&&\qquad =b(\boldsymbol{u}, \boldsymbol{v})-b_N(\boldsymbol{u}_N^h,
\boldsymbol{v}^h) -b_N(\boldsymbol{u}_N^h, \boldsymbol{v}-\boldsymbol{v}^h)
+\int_{\partial
B_R}\left(\mathscr{T}-\mathscr{T}_N\right)\boldsymbol{u}_N^h\cdot
\overline{\boldsymbol{v}}{\rm d}s  \\
&&\qquad\qquad +\int_{\partial
B_R}\left(\mathscr{T}-\mathscr{T}_N\right)\boldsymbol{\xi}\cdot
\overline{\boldsymbol{v}}{\rm d}s  \\
&&\qquad =-b_N(\boldsymbol{u}_N^h, \boldsymbol{v}-\boldsymbol{v}^h)
+\int_{\partial B_R}\left(\mathscr{T}-\mathscr{T}_N\right)\boldsymbol{u}\cdot
\overline{\boldsymbol{v}}{\rm d}s.
\end{eqnarray*}

For any  $\boldsymbol v^h\in \boldsymbol V_{h, \partial D}$, it follows
from the integration by parts that 
\begin{eqnarray}\label{lemma_xi_3_s1}
&& -b_N(\boldsymbol{u}_N^h, \boldsymbol{v}-\boldsymbol{v}^h)\notag\\
&=& -\sum\limits_{K\in\mathcal M_h}\left\{ \mu\int_{K}
\nabla\boldsymbol{u}_N^h:\nabla 
\big(\overline{\boldsymbol v}-\overline{\boldsymbol{v}^h}\big){\rm
d}\boldsymbol{x} +(\lambda+\mu)\int_{K}
\big(\nabla\cdot\boldsymbol{u}_N^h\big) \nabla\cdot\big(\overline{\boldsymbol
v}-\overline{\boldsymbol v}^h\big){\rm d}\boldsymbol{x}\right\}\notag\\
&& -\sum\limits_{K\in\mathcal
M_h}\left\{-\omega^2\int_{K}\boldsymbol{u}_N^h\cdot 
\big(\overline{\boldsymbol v}-\overline{\boldsymbol{v}^h}\big){\rm
d}\boldsymbol{x} -\int_{\partial B_R\cap \partial K}
\mathscr{T}\boldsymbol{u}_N^{h}\cdot
\overline{\boldsymbol v}-\overline{\boldsymbol{v}^h}{\rm d}s\right\}\notag\\
&=& \sum\limits_{K\in\mathcal M_h} \left\{ -\int_{\partial
K}\left[\mu\nabla\boldsymbol{u}_N^h\cdot\boldsymbol{\nu}
+(\lambda+\mu)\nabla(\nabla\cdot\boldsymbol{u}_N^h)\cdot\boldsymbol{\nu}\right]
\cdot\big(\overline{\boldsymbol v}-\overline{\boldsymbol{v}^h}\big)
{\rm d}\boldsymbol{x} +\int_{\partial B_R\cap \partial
K}\mathscr{T}\boldsymbol{u}_N^h\cdot
\big(\overline{\boldsymbol v}-\overline{\boldsymbol{v}^h}\big){\rm
d}s\right\}\notag\\
&& +\sum\limits_{K\in\mathcal M_h}\int_{K}
\left[\mu\Delta\boldsymbol{u}_N^h
+(\lambda+\mu)\nabla\nabla\cdot\boldsymbol{u}_N^h+\omega^2\boldsymbol{u}
_N^h\right]\cdot\big( \overline{\boldsymbol
v}-\overline{\boldsymbol{v}^h}\big){\rm d}\boldsymbol{x}\notag\\
&=&  \sum\limits_{K\in\mathcal M_h}\left[\int_{K}\mathscr R\boldsymbol{u}_N^h
\cdot \big( \overline{\boldsymbol
v}-\overline{\boldsymbol{v}^h}\big){\rm d}\boldsymbol{x}
+\sum\limits_{e\in \partial K}\frac{1}{2}\int_{e}J_e
\cdot  \big(\overline{\boldsymbol v}-\overline{\boldsymbol{v}^h}\big){\rm ds} 
\right].
\end{eqnarray}
We take $\boldsymbol{v}^h=\Pi_h \boldsymbol{v}\in\boldsymbol V_{h, \partial D}$,
where $\Pi_h$ is the Scott--Zhang interpolation operator \cite{SZ_mc90}, which
has the following interpolation estimates
\begin{eqnarray*}
\|\boldsymbol{v}-\Pi_h
\boldsymbol{v}\|_{\boldsymbol{L}^2(K)}\lesssim h_K
\|\nabla\boldsymbol{v}\|_{\boldsymbol{L}^2(\tilde{K})},\quad
\|\boldsymbol{v}-\Pi_h\boldsymbol{v}\|_{\boldsymbol{L}^2(e)}\lesssim
h_e^{1/2}\|\boldsymbol{v}\|_{\boldsymbol{H}^{1}(\tilde{K}_e)}.
\end{eqnarray*}
Here $\tilde{K}$ and $\tilde{K}_e$ are the union of all the triangular elements
in $\mathcal M_h$, which have nonempty intersection with the element $K$
and the side $e$, respectively. Using the H\"{o}lder inequality in
\eqref{lemma_xi_3_s1}, we get
\begin{eqnarray*}
\left|b_N(\boldsymbol{u}_N^h, \boldsymbol{v}-\boldsymbol{v}^h)\right|
\lesssim \left(\sum\limits_{K\in\mathcal
M_h}\eta_K^2\right)^{1/2}\|\boldsymbol{v}\|_{ \boldsymbol{H}^1(\Omega)}.
\end{eqnarray*}
The proof is completed after combining Lemma \ref{lemma_xi_5}.
\end{proof}

In the following lemma, we estimate the last two terms in \eqref{LemmaXi_1}.
\begin{lemma}\label{lemma_xi_6}
For any $\delta>0$, there exists a positive constant $C(\delta)$ independent of
$N$ such that
\begin{equation*}
\Re\int_{\partial B_R}\mathscr{T}_N\boldsymbol{\xi}\cdot \overline{
\boldsymbol{\xi}}{\rm d}s\leq C(\delta)
\|\boldsymbol{\xi}\|^2_{\boldsymbol{L}^2(B_R\setminus B_{\hat R})}
+\bigg(\frac{R}{\hat R}\bigg)\delta\|\boldsymbol{\xi}\|^2_{\boldsymbol{H}^{1}(B_R\setminus
B_{\hat R})}.                                                   
\end{equation*}
\end{lemma}

\begin{proof}
Using \eqref{Truncated_TBC}, we get from a simple calculation that
\[
\Re\int_{\partial B_R}\mathscr{T}_N\boldsymbol{\xi}\cdot
\overline{ \boldsymbol{\xi}}{\rm d}s
= 2\pi R\,\Re\sum\limits_{|n|\leq N
}\left(M_n\boldsymbol{\xi}_n\right)\cdot\overline{\boldsymbol{\xi}_n}.
\]
Denote $\hat{M}_n=(M_n+M_n^*)/2$. Then
$\Re\left(M_n\boldsymbol{\xi}_n\right)\cdot 
\overline{\boldsymbol{\xi}_n}=\big(\hat{M}_n\boldsymbol{\xi}
_n\big)\cdot\overline{\boldsymbol{\xi}_n}$. It is shown in \cite{LWWZ-ip16}
that $\hat{M}_n$ is negative definite for sufficiently large $|n|$, i.e., there
exists $N_0>0$ such that
$\big(\hat{M}_n\boldsymbol{\xi}_n\big)\cdot\overline{\boldsymbol{\xi}
_n}\leq 0$ for any $|n|>N_0$. Hence
\begin{equation}\label{lemma_xi_6-s1}
\Re\int_{\partial B_R}\mathscr{T}_N\boldsymbol{\xi}\cdot
\overline{ \boldsymbol{\xi}}{\rm d}s
=2\pi R\sum\limits_{|n|\leq \min(N_0,
N)}\big(\hat{M}_n\boldsymbol{\xi}_n\big)\cdot\overline{\boldsymbol{\xi}_n}
+2\pi R\sum\limits_{N\geq |n|>\min(N_0,
N)}\big(\hat{M}_n\boldsymbol{\xi}_n\big)\cdot\overline{\boldsymbol{\xi}_n}
\end{equation}
Here we define  
\[
\sum\limits_{N>|n|>\min(N_0, N)}\big(\hat{M}_n\boldsymbol{\xi}
_n\big)\cdot\overline{\boldsymbol{\xi}_n} =0, \quad N>N_0.
\]
Since the second part in \eqref{lemma_xi_6-s1} is non-positive, we only need to
estimate the first part which consists of finite terms. Moreover we have
\begin{eqnarray*}
 \Re\int_{\partial B_R}\mathscr{T}_N\boldsymbol{\xi}\cdot
\overline{ \boldsymbol{\xi}}{\rm d}s
&\leq& 2\pi R\sum\limits_{|n|\leq \min(N_0,
N)}\big(\hat{M}_n\boldsymbol{\xi}_n\big)\cdot\overline{\boldsymbol{\xi}_n}\\
&\leq & C \sum\limits_{|n|\leq \min(N_0, N)}|\boldsymbol{\xi}_n|^2\leq
C\|\boldsymbol{\xi}\|^2_{\boldsymbol{L}^2(\partial B_R)}.
\end{eqnarray*}

Consider the annulus 
\[
 B_R\setminus B_{\hat{R}}=\{(r, \theta): \hat{R}<r<R,\,0<\theta<2\pi\}. 
\]
For any $\delta>0$, it follows from Young's inequality that 
\begin{eqnarray*}
&& (R-\hat R)|u(R)|^2 = \int_{\hat R}^{R}
|u(r)|^2{\rm d}r+\int_{\hat R}^{R}\int_{r}^{R}\frac{\rm d}{{\rm d}t}|u(t)|^2{\rm
d}t{\rm d}r\\
&&\qquad
\leq \int_{\hat R}^{R}|u(r)|^2{\rm d}r+(R-\hat R)\int_{\hat R}^{R}
2|u(r)||u'(r)|{\rm d}r \\
&&\qquad =\int_{\hat R}^{R} |u(r)|^2{\rm d}r+(R-\hat
R)\int_{\hat R}^{R}2\frac{|u(r)|}{\sqrt{\delta}}\sqrt{
\delta}|u'(r)|{\rm d}r\\
&&\qquad \leq
\int_{\hat
R}^{R}|u(r)|^2{\rm d}r+\delta^{-1}(R-\hat R)\int_{\hat
R}^{R}|u(r)|^2 {\rm d}r +\delta(R-\hat R)\int_{\hat R}^{R}|u'(r)|^2{\rm d}r,
\end{eqnarray*}
which gives 
\[
|u(R)|^2\leq
\left[\delta^{-1}+(R-\hat R)^{-1}\right]\int_{\hat R}^{R}|u(r)|^2
+\delta\int_{\hat R}^{R}|u'(r)|^2{\rm d}r.
\]
On the other hand, we have 
\begin{eqnarray*}
\|\nabla u\|^2_{\boldsymbol{L}^2(B_R\setminus B_{\hat R})} &=&
2\pi\sum\limits_{n\in\mathbb{Z}}\int_{\hat R}^{R}\Big(
r|u_n'(r)|^2+\frac{n^2}{r}|u_n(r)|^2\Big){\rm d}r,\\
\|u\|_{L^{2}(B_R\setminus B_{\hat R})}^2  &=&
2\pi\sum\limits_{n\in\mathbb{Z}}\int_{\hat R}^{R}r|u_n(r)|^2{\rm d}r.	
\end{eqnarray*}
Using the above estimates, we have for any $u\in H^{1}(B_R\setminus B_{\hat R})$ that 
\begin{eqnarray*}
&& \|u\|^2_{L^2(\partial B_R)} = 2\pi
R\sum\limits_{n\in\mathbb{Z}}|u_n(R)|^2 \\
&&\,\leq 2\pi
R\left[\delta^{-1}+(R-\hat
R)^{-1}\right]\sum\limits_{n\in\mathbb{Z}}\int_{\hat R}^{R}|u_n(r)|^2
+2\pi
R\delta\sum\limits_{n\in\mathbb{Z}}\int_{\hat R}^{R}|u'(r)|^2{\rm d}r\\
&&\,\leq  2\pi \left[\delta^{-1}+(R-\hat R)^{-1}\right]
\bigg(\frac{R}{\hat R}\bigg)\sum\limits_{n\in\mathbb{Z}}\int_{\hat R}^{R}
r|u_n(r)|^2{\rm d}r +2\pi
\delta\bigg(\frac{R}{\hat R}\bigg)\sum\limits_{n\in\mathbb{Z}}\int_{\hat R}^{R} \Big(r|u'_n(r)|^2 +\frac{n^2}{r}|u_n(r)|^2\Big){\rm d}r\\
&&\,\leq 2\pi \left[\delta^{-1}+(R-\hat R)^{-1}\right]
\bigg(\frac{R}{\hat R}\bigg)\|u\|^2_{L^2(B_R\setminus B_{\hat R})}
+\delta\bigg(\frac{R}{\hat R}\bigg)\|\nabla u\|^2_{L^2(B_R\setminus B_{\hat R})}\\
&&\, \leq C(\delta) \|u\|^2_{L^2(B_R\setminus
B_{\hat R})}+\bigg(\frac{R}{\hat R}\bigg)\delta\|\nabla u\|^2_{L^2(B_R\setminus
B_{\hat R})}.	
\end{eqnarray*}
Therefore,
\begin{eqnarray*}
\Re\int_{\partial B_R}\mathscr{T}_N\boldsymbol{\xi}\cdot
\overline{ \boldsymbol{\xi}}{\rm d}s &\leq&
C\|\boldsymbol{\xi}\|^2_{\boldsymbol{L}^2(\partial B_R)}
\leq C(\delta) 
\|\boldsymbol{\xi}\|^2_{\boldsymbol{L}^2(B_R\setminus B_{\hat R})}
+\bigg(\frac{R}{\hat R}\bigg)\delta\int_{\Omega}|\nabla\boldsymbol{\xi}|            
{\rm d}\boldsymbol{x}\\
&\leq& C(\delta) \|\boldsymbol{\xi}\|^2_{\boldsymbol{L}^2(B_R\setminus B_{\hat R})} +\bigg(\frac{R}{\hat R}\bigg)\delta\|\boldsymbol{\xi}\|^2_{\boldsymbol{H}^{1}(B_R\setminus B_{\hat R})},
\end{eqnarray*}
which completes the proof. 
\end{proof}

To estimate the third term on the right hand side of \eqref{LemmaXi_1}, we
consider the dual problem  
\begin{equation}\label{dp}
b(\boldsymbol{v}, \boldsymbol{p})=\int_{\Omega}
\boldsymbol{v}\cdot\overline{\boldsymbol{\xi}}{\rm d}\boldsymbol{x},
\quad \forall \boldsymbol{v}\in \boldsymbol{H}^{1}_{\partial
D}(\Omega).
\end{equation}
It is easy to check that $\boldsymbol{p}$ is the solution of the following
boundary value problem
\begin{equation}\label{Dualproblem}
\begin{cases}
\mu\Delta
\boldsymbol{p}+(\lambda+\mu)\nabla
\nabla\cdot\boldsymbol{p}+\omega^2\boldsymbol{p} = -\boldsymbol{\xi}
\quad &\text{in }\Omega,\\
 \boldsymbol{p}=0 \quad &\text{on }\partial D,\\
 \mathscr B\boldsymbol{p}=\mathscr{T}^{*}\boldsymbol{p} \quad &
\text{on }\partial B_R,
\end{cases}
\end{equation}
where $\mathscr T^*$ is the adjoint operator to the DtN operator $\mathscr T$.
Letting $\boldsymbol v=\boldsymbol\xi$ in \eqref{dp}, we obtain
\begin{equation}\label{LemmaXi_3a}
\|\boldsymbol{\xi}\|^2_{\boldsymbol{L}^2(\Omega)}
=b(\boldsymbol{\xi}, \boldsymbol{p})+\int_{\partial
B_R}\left(\mathscr{T} -\mathscr{T}_N\right)\boldsymbol{\xi}\cdot\overline{
\boldsymbol{p}}{\rm d}s -\int_{\partial
B_R}\left(\mathscr{T}-\mathscr{T}_N\right)\boldsymbol{\xi}\cdot\overline{
\boldsymbol{p}}{\rm d}s.
\end{equation}	

To evaluate \eqref{LemmaXi_3a}, we need to explicitly solve system
\eqref{Dualproblem}, which is very complicate due to the coupling of the
compressional and shear wave components. We consider the Helmholtz
decomposition to the boundary value problem \eqref{Dualproblem}. Let
\[
\boldsymbol{\xi}=\nabla\xi_1+{\bf curl}\xi_2,
\]
where $\xi_j, j=1, 2$ has the Fourier series expansion 
\[
\xi_j(r, \theta)=\sum_{n\in\mathbb Z}\xi_{jn}(r) e^{{\rm i}n\theta},\quad
\hat{R}<r<R.
\]
Meanwhile, we assume that 
\begin{equation}\label{xi_pfc}
\boldsymbol{\xi}(r,\theta)=\sum\limits_{n\in\mathbb{Z}}\left(\xi_n^{r}(r)
\boldsymbol{e}_r +\xi_n^{\theta}(r)\boldsymbol{e}_{\theta}\right)e^{{\rm
i}n\theta}.
\end{equation}
Using the Fourier series expansions and the Helmholtz decomposition, we
get 
\begin{eqnarray*}
\boldsymbol{\xi} (r,\theta) &=&
\sum\limits_{n\in\mathbb{Z}}\left[\xi_n^{r}(r)\boldsymbol{e}_r
+\xi_n^{\theta}(r)\boldsymbol{e}_{\theta}\right]e^{{\rm
i}n\theta}=\nabla\xi_1+{\bf curl}\xi_2\\
&=& \sum\limits_{n\in\mathbb{Z}}
\left[\xi'_{1n}(r)\boldsymbol{e}_r+\frac{{\rm
i}n}{r}\xi_{1n}(r)\boldsymbol{e}_{\theta}
+\frac{{\rm i}n}{r}\xi_{2n}(r)\boldsymbol{e}_r-\xi'_{2n}(r)
\boldsymbol{e}_{\theta}\right]e^{{\rm i}n\theta}\\
&=&  \sum\limits_{n\in\mathbb{Z}}\left[
\Big(\xi'_{1n}(r)+\frac{{\rm i}n}{r}\xi_{2n}(r)\Big)\boldsymbol{e}_r
+\Big(\frac{{\rm
i}n}{r}\xi_{1n}(r)-\xi'_{2n}(r)\Big)\boldsymbol{e}_{\theta}\right] e^{{\rm
i}n\theta},
\end{eqnarray*}
which shows that the Fourier coefficients $\xi_{1n}, \xi_{2n}$ satisfy
\begin{eqnarray*}
\xi'_{1n}(r)+\frac{{\rm i}n}{r}\xi_{2n}(r)=\xi_n^{r}(r),\quad 
\frac{{\rm i}n}{r}\xi_{1n}(r)-\xi'_{2n}(r)
=\xi_n^{\theta}(r), \quad & r\in (\hat R, R).
\end{eqnarray*} 

\begin{lemma}\label{Lemma_Xi_System}
The Fourier coefficients $\xi_{jn}, j=1, 2$ satisfy the system
\begin{equation}\label{Dual_xi}
\begin{cases}
\xi'_{1n}(r)+\frac{{\rm i}n}{r}\xi_{2n}(r)=\xi_n^{r}(r), \quad & r\in(\hat R,
R),\\
\frac{{\rm i}n}{r}\xi_{1n}(r)-\xi'_{2n}(r)
=\xi_n^{\theta}(r), \qquad & r\in(\hat R, R),\\
\xi_{1n}(R)=0,\quad \xi_{2n}(R)=0,\quad & r=R,
\end{cases}
\end{equation} 
which has a unique solution given by 
\begin{eqnarray}
\xi_{1n}(r) &=& -\frac{1}{2}\int_{r}^{R}
\left[\left(\frac{r}{t}\right)^n
+\left(\frac{t}{r}\right)^n\right]\xi_n^r(t){\rm d}t
+\frac{{\rm i}}{2}\int_{r}^{R}\left[\left(\frac{r}{t}\right)^n
-\left(\frac{t}{r}\right)^n\right]\xi_n^{\theta}(t){\rm d}t,
\label{SoluXi_1}\\
\xi_{2n}(r) &=& \frac{{\rm i}}{2}\int_{r}^{R}
\left[\left(\frac{t}{r}\right)^n
-\left(\frac{r}{t}\right)^n\right]\xi_n^r(t){\rm d}t
-\frac{1}{2}\int_{r}^{R}\left[\left(\frac{r}{t}\right)^n
+\left(\frac{t}{r}\right)^n\right]\xi_n^{\theta}(t){\rm d}t.
\label{SoluXi_2}					
\end{eqnarray}
\end{lemma}

\begin{proof}
Denote 
\[
A_n(r)=\begin{bmatrix} 0 & -\frac{{\rm i}n}{r}\\
\frac{{\rm i}n}{r} & 0\end{bmatrix}.
\]
By the standard theory of the first order differential system, the fundamental
solution $\Phi_n(r)$ is
\begin{eqnarray*}
\Phi_n(r)&=&e^{\int_{\hat R}^r A_n(\tau){\rm d}\tau}=\exp \left(
\begin{bmatrix}
0 & -{\rm i}n\ln\frac{r}{\hat R}\\
{\rm i}n\ln\frac{r}{\hat R} & 0
\end{bmatrix}\right)\\
&=& \begin{bmatrix}
\frac{1}{\sqrt{2}} & \frac{\rm i}{\sqrt{2}}\\
\frac{\rm i}{\sqrt{2}} & \frac{1}{\sqrt{2}}
\end{bmatrix}
\begin{bmatrix}
\left(\frac{r}{\hat R}\right)^n & 0\\
0 & \left(\frac{r}{\hat R}\right)^{-n}
\end{bmatrix}
\begin{bmatrix}
\frac{1}{\sqrt{2}} & -\frac{\rm i}{\sqrt{2}}\\
-\frac{\rm i}{\sqrt{2}} & \frac{1}{\sqrt{2}}
\end{bmatrix}.
\end{eqnarray*}
The inverse of $\Phi_n$ is
\begin{equation*}
\Phi_n^{-1}(r)=
\begin{bmatrix}
\frac{1}{\sqrt{2}} & \frac{\rm i}{\sqrt{2}}\\
\frac{\rm i}{\sqrt{2}} & \frac{1}{\sqrt{2}}
\end{bmatrix}
\begin{bmatrix}
\left(\frac{r}{\hat R}\right)^{-n} & 0\\
0 & \left(\frac{r}{\hat R}\right)^{n}
\end{bmatrix}
\begin{bmatrix}
\frac{1}{\sqrt{2}} & -\frac{\rm i}{\sqrt{2}}\\
-\frac{\rm i}{\sqrt{2}} & \frac{1}{\sqrt{2}}
\end{bmatrix}.
\end{equation*}

Using the method of variation of parameters, we let 
\[
(\xi_{1n}(r),\xi_{2n}(r))^\top=\Phi_n(r)C_n(r),
\]
where the unknown vector $C_n(r)$ satisfies 
\begin{eqnarray}\label{Lemma_Xi_System-s1}
C_n^\prime (r) &=& \Phi_n^{-1}(r)(\xi_n^r(r),
\xi_n^{\theta}(r))^\top\notag\\
&=&\frac{1}{2}
\begin{bmatrix}
\left[\left(\frac{r}{\hat R}\right)^{-n}+\left(\frac{r}{\hat R}\right)^n\right]
\xi_n^r(r)
+{\rm i}\left[\left(\frac{r}{\hat
R}\right)^{n}-\left(\frac{r}{\hat R}\right)^{-n} \right] \xi_n^{\theta}(r)\\[8pt]
{\rm
i}\left[\left(\frac{r}{\hat
R}\right)^{-n}-\left(\frac{r}{\hat R}\right)^n\right] \xi_n^r(r)
+\left[\left(\frac{r}{\hat
R}\right)^{-n}+\left(\frac{r}{\hat R}\right)^{n}\right] \xi_n^{\theta}(r)
\end{bmatrix}.
\end{eqnarray}
Using the boundary condition yields
\[
(\xi_{1n}(R), \xi_{2n}(R))^\top
=\Phi_n(R)C_n(R)=(0,0)^\top,
\]
which implies that $C_n(R)=(0, 0)^\top$. Then
\begin{equation}\label{Lemma_Xi_System-s2}
C_n(r)=-\int_{r}^{R}C_n^\prime (t){\rm d}t.
\end{equation}
Combining \eqref{Lemma_Xi_System-s1} and \eqref{Lemma_Xi_System-s2}, we have
\begin{equation*}
C_n(r)=-\frac{1}{2}\begin{bmatrix}
\int_{r}^{R}\left[\left(\frac{t}{\hat R}\right)^{-n}+\left(\frac{t}{\hat R}
\right)^{n}\right]\xi_n^{r}(t){\rm d}t
+{\rm i}\int_{r}^{R}\left[\left(\frac{t}{\hat
R}\right)^{n}-\left(\frac{t}{\hat R}
\right)^{-n}\right]\xi_n^{\theta}(t){\rm d}t\\[8pt]
{\rm i}\int_{r}^{R}\left[\left(\frac{t}{\hat
R}\right)^{-n}-\left(\frac{t}{\hat R} \right)^{n}\right]\xi_n^{r}(t){\rm d}t
+\int_{r}^{R}\left[\left(\frac{t}{\hat R}\right)^{n}+\left(\frac{t}{\hat R}
\right)^{-n}\right]\xi_n^{\theta}(t){\rm d}t.
\end{bmatrix}.
\end{equation*}
Substituting $C_n(r)$ into the general solution, we obtain 
\begin{eqnarray*}
\xi_{1n}(r) 
&=&-\frac{1}{2}\left(\frac{r}{\hat
R}\right)^n\int_{r}^{R}\left(\frac{t}{\hat R} \right)^{-n}\xi_n^{r}(t){\rm d}t
+\frac{\rm i}{2}\left(\frac{r}{\hat
R}\right)^n\int_{r}^{R}\left(\frac{t}{\hat R} \right)^{-n}
\xi_n^{\theta}(t){\rm d}t\\
&&\qquad
-\frac{1}{2}\left(\frac{r}{\hat
R}\right)^{-n}\int_{r}^{R}\left(\frac{t}{\hat R} \right)^{n}\xi_n^{r}(t){\rm d}t
-\frac{\rm i}{2}
\left(\frac{r}{\hat R}\right)^{-n}\int_{r}^{R}\left(\frac{t}{\hat
R}\right)^{n} \xi_n^{\theta}(t){\rm d}t\\
\xi_{2n}(r) 
&=&-\frac{\rm
i}{2}\left(\frac{r}{\hat R}\right)^n\int_{r}^{R}\left(\frac{t}{\hat R}
\right)^{-n} \xi_n^{r}(t){\rm d}t +\frac{\rm i}{2}
\left(\frac{r}{\hat R}\right)^{-n}\int_{r}^{R}\left(\frac{t}{\hat
R}\right)^{n} \xi_n^{r}(t){\rm d}t\\
&&\qquad
-\frac{1}{2}\left(\frac{r}{\hat R}\right)^{n}\int_{r}^{R}\left(\frac{t}{\hat
R} \right)^{-n}\xi_n^{\theta}(t){\rm d}t-\frac{1}{2}
\left(\frac{r}{\hat R}\right)^{-n}\int_{r}^{R}\left(\frac{t}{\hat
R}\right)^{n} \xi_n^{\theta}(t){\rm d}t,
\end{eqnarray*}		
which completes the proof.
\end{proof}

Let $\boldsymbol p$ be the solution of the dual problem \eqref{Dualproblem}.
Then $\boldsymbol p$ satisfies the following boundary value problem in
$B_R\setminus\overline B_{\hat{R}}$:
\begin{equation}\label{Dual_p}
\begin{cases}
\mu\Delta\boldsymbol{p}+(\lambda+\mu)\nabla\nabla\cdot\boldsymbol{p}
+\omega^2 \boldsymbol{p}=-\boldsymbol{\xi}
\quad & \text{in}\, B_R\setminus\overline{B_{\hat R}}, \\
 \boldsymbol{p}(\hat R,\theta)=\boldsymbol{p}(\hat R,
\theta) \quad & \text{on}\, \partial B_{\hat R},\\
\mathscr B\boldsymbol{p}=\mathscr{T}^{*}\boldsymbol{p} \quad
& \text{on}\,\partial B_R.	
\end{cases}
\end{equation}
Introduce the Helmholtz decomposition for $\boldsymbol p$:
\begin{equation}\label{hd_pq}
 \boldsymbol p=\nabla q_1+{\bf curl} q_2,
\end{equation}
where $q_j, j=1, 2$ admits the Fourier series expansion
\[
 q_j(r, \theta)=\sum\limits_{n\in\mathbb{Z}}q_{jn}(r)e^{{\rm
i}n\theta}.
\]

Let $\xi_{jn}, j=1, 2$ be the solution of the system \eqref{Dual_xi}. Consider
the second order system for $q_{jn}, j=1, 2$:
\begin{equation}\label{Dual_phi}
 \begin{cases}
  q''_{jn}(r)+\frac{1}{r}q'_{jn}(r)+\big(\kappa_j^2-\left(\frac{n}{R}
\right)^2\big)q_{jn}(r) =c_j\xi_{jn}(r),\quad &
r\in(\hat R, R),\\
q_{jn}(\hat R)=q_{jn}(\hat R),\quad & r=\hat R,\\
q'_{jn}(R)=\overline{\alpha_{jn}} q_{jn}(R),\quad & r=R,
 \end{cases}
\end{equation}
where $c_1=-1/(\lambda+2\mu), c_2=-1/\mu,$ and $\alpha_{jn}$ is given
in \eqref{alpha_jn}. The boundary
condition $q'_{jn}(R)=\overline{\alpha_{jn}} q_{jn}(R)$ comes
from \eqref{ptbc}, i.e., $q_j$ satisfies the boundary condition 
\[
 \partial_r q_j=\mathscr T_j^* q_j:=\sum_{n\in\mathbb
Z}\overline{\alpha_{jn}} q_{jn}(R)e^{{\rm i}n\theta}\quad\text{on}\,\partial
B_R,
\]
where $\mathscr T_j^*$ is the adjoint operator to the DtN operator $\mathscr
T_j$.

\begin{lemma}\label{Thm_Xi}
The boundary value problem \eqref{Dual_p} and the second order system
\eqref{Dual_phi} are equivalent under the Helmholtz decomposition \eqref{hd_pq}.
\end{lemma}

\begin{proof}
It suffices to show if the Fourier coefficients $q_{jn}$ satisfy the second
order system \eqref{Dual_phi}, then $\boldsymbol p=\nabla q_1+{\bf curl}q_2$ is
the solution of \eqref{Dual_p}.

In the polar coordinates, we let 
\begin{equation}\label{Thm_Xi-s1}
\boldsymbol p(r, \theta)=\sum_{n\in\mathbb Z}(p_n^r(r)\boldsymbol e_r +
p_n^\theta(r)\boldsymbol e_\theta)e^{{\rm i}n\theta},\quad r\in(\hat R, R). 
\end{equation}
It follows from the Helmholtz decomposition that
\begin{equation}\label{Thm_Xi-s2}
p_n^{r}(r)=q'_{1n}(r)+\frac{{\rm i}n}{r}q_{2n}(r),\quad 
p_n^{\theta}(r)=\frac{{\rm i}n}{r}q_{1n}(r)-q'_{2n}(r).
\end{equation}
Using \eqref{Thm_Xi-s1}--\eqref{Thm_Xi-s2}, we have from a straightforward
calculation that
\begin{eqnarray*}
\mathscr B\boldsymbol{p} &=&
\big(\mu\partial_r\boldsymbol{p}+(\lambda+\mu)\nabla\cdot\boldsymbol{p}
\boldsymbol {e}_r\big) |_{r=R} \\
&=&\sum\limits_{n\in\mathbb{Z}}
\Big[(\lambda+2\mu)q''_{1n}(R)+(\lambda+\mu)\frac{1}{R}q'_{1n}(R)
-(\lambda+\mu)\frac{n^2}{R^2}q_{1n}(R)\Big]e^{{\rm
i}n\theta}\boldsymbol{e}_r\\
&&+\sum\limits_{n\in\mathbb{Z}}\Big[\mu\frac{{\rm
i}n}{R}q'_{1n}(R)-\mu\frac{{\rm i}n}{R^2}q_{1n}(R)\Big]
e^{{\rm i}n\theta}\boldsymbol{e}_{\theta}
+\sum\limits_{n\in\mathbb{Z}}\Big[\mu\frac{{\rm
i}n}{R}q'_{2n}(R)-\mu\frac{{\rm i}n}{R^2}q_{2n}(R)\Big]
e^{{\rm i}n\theta}\boldsymbol{e}_{r}\\
&&+\sum\limits_{n\in\mathbb{Z}}-\mu q''_{2n}(R)e^{{\rm
i}n\theta}\boldsymbol{e}_{\theta}.
\end{eqnarray*}
On the other hand, it is easy to verify that 
\begin{eqnarray*}
\mathscr{T}^{*}\boldsymbol{p} &=& \sum\limits_{n\in\mathbb{Z}}
\left\{\Big[\overline{M_{11}^{(n)}} p_n^r(R) +\overline{M_{21}^{(n)}}
p_n^{\theta}(R)\Big]\boldsymbol{e}_r +\Big[\overline{M_{12}^{(n)}} p_n^r(R)
+\overline{M_{22}^{(n)}} p_n^{\theta}(R)\Big]\boldsymbol{e}_{\theta}\right\}e^{{
\rm i}n\theta}\\
&=& \sum\limits_{n\in\mathbb{Z}}\left\{\overline{M_{11}^{(n)}}
\Big[q'_{1n}(R)+\frac{{\rm i}n}{R}q_{2n}(R)\Big]
+\overline{M_{21}^{(n)}} \Big[\frac{{\rm
i}n}{R}q_{1n}(R)-q'_{2n}(R)\Big]\right\}\boldsymbol{e}_r e^{{\rm
i}n\theta}\\
&&\quad +\sum\limits_{n\in\mathbb{Z}} \left\{\overline{M_{12}^{(n)}}
\Big[q'_{1n}(R)+\frac{{\rm i}n}{R}q_{2n}(R)\Big]
+\overline{M_{22}^{(n)}} \Big[\frac{{\rm
i}n}{R}q_{1n}(R)-q'_{2n}(R)\Big]\right\}\boldsymbol{e}_{\theta} e^{{\rm
i}n\theta}, 
\end{eqnarray*}
where $M_{ij}^{(n)}, i,j=1, 2$ are given in \eqref{Mn}.

Using the boundary condition $q'_{jn}(R)=\overline{\alpha_{jn}} q_{jn}(R)$, we
get 
\begin{eqnarray*}
&& \left(\mu\frac{{\rm
i}n}{R}-\overline{M_{12}^{(n)}}\right)q'_{1n}(R)-\left(\overline{M_{22}^{(n)}}
\frac{{\rm i}n}{R} +\mu\frac{{\rm i}n}{R^2}\right)q_{1n}(R)\\
&=& \left(\mu\frac{{\rm i}n}{R}-{\rm
i}n\frac{\mu}{R}+\omega^2\frac{{\rm i}n}{R}\frac{1}{\overline{\Lambda_n(R)}}
\right)q'_{1n}(R)
-\left(-\frac{\mu}{R}\frac{{\rm i}n}{R}+\omega^2\frac{{\rm
i}n}{R}\frac{\overline{\alpha_{1n}}}{\overline{\Lambda_n(R)}}
+\mu\frac{{\rm i}n}{R^2}\right)q_{1n}(R)\\
&=& \omega^2 \frac{{\rm
i}n}{R}\frac{1}{\overline{\Lambda_n(R)}}\left(q'_{1n}(R)-\overline{ \alpha_{1n}}
q_{1n}(R)\right)=0
\end{eqnarray*}
and 
\begin{eqnarray*}
&& \left(\mu\frac{{\rm
i}n}{R}+\overline{M_{21}^{(n)}}\right)q'_{2n}(R)-\left(\overline{M_{11}^{(n)}}
\frac{{\rm i}n}{R}+\mu\frac{{\rm i}n}{R^2}\right)q_{2n}(R)\\
&=& \left(\mu\frac{{\rm i}n}{R}-{\rm i}n\frac{\mu}{R}+{\rm
i}n\frac{\omega^2}{R}\frac{1}{\overline{\Lambda_n(R)}}\right)q'_{2n}(R)
-\left(\mu\frac{{\rm i}n}{R^2}-\frac{\mu}{R}\frac{{\rm
i}n}{R}+\omega^2\frac{{\rm
i}n}{R}\frac{\overline{\alpha_{2n}}}{\overline{\Lambda_n(R)}}
\right)q_{2n}(R)\\
&=&  {\rm i}n\frac{\omega^2}{R}\frac{1}{\overline{\Lambda_n(R)}}\left(
q'_{2n}(R)-\overline{\alpha_{2n}}q_{2n}(R)\right)=0.	
\end{eqnarray*}
Since $q_{2n}$ satisfies the second order equation
\[
q''_{2n}(r)+\frac{1}{r}q'_{2n}(r)+\left(\kappa_2^2-\left(\frac{n}{R}
\right)^2\right)q_{2n}(r)=-\frac{1}{\mu}\xi_{2n},\quad r\in (\hat R, R),
\]
we obtain from the boundary condition $\xi_{2n}(R)=0$ that 
\begin{eqnarray*}
&&-\mu q''_{2n}(R)-\left(\overline{M_{12}^{(n)}}\frac{{\rm
i}n}{R}q_{2n}(R)-\overline{M_{22}^{(n)}}q'_{2n}(R)\right)\\
&&\quad= -\mu q''_{2n}(R)-\frac{{\rm i}n}{R}\left({\rm
i}n\frac{\mu}{R} -\omega^2 \frac{{\rm
i}n}{R}\frac{1}{\overline{\Lambda_n(R)}}\right)q_{2n}(R)
+\left(-\frac{\mu}{R}+\omega^2
\frac{\overline{\alpha_{1n}}}{\overline{\Lambda_n(R)}}\right)q'_{2n}(R)\\
&&\quad= \xi_{2n}(R)+\mu\kappa_2^2 q_{2n}(R)
+\omega^2\left(\frac{{\rm i}n}{R}\right)^2\frac{1}{\overline{\Lambda_n(R)}}
q_{2n}(R) +\omega^2\frac{\overline{\alpha_{1n}}}{\overline{\Lambda_n(R)}}
q'_{2n}(R)\\
&&\quad=\xi_{2n}(R)+\frac{\omega^2}{\overline{\Lambda_n(R)}}\left(
\left(\frac{n}{R} \right)^2 q_{2n}(R)
-\overline{\alpha_{1n}}\overline{\alpha_{2n}} q_{2n}(R)
-\left(\frac{n}{R}\right)^2
q_{2n}(R)+\overline{\alpha_{1n}} q'_{2n}(R)\right)\\
&&\quad= \xi_{2n}(R)+\omega^2\frac{\overline{\alpha_{1n}}}{\overline{
\Lambda_n(R)}} \left(-\overline{\alpha_{2n}} q_{2n}(R)+q'_{2n}(R)\right)
=0.
\end{eqnarray*}
Similarly, combining the equation
\[
q''_{1n}(r)+\frac{1}{r}q'_{1n}(r)+\left(\kappa_1^2-\left(\frac{n}{r}
\right)^2\right)q_{1n}(r)=-\frac{1}{\lambda+2\mu}\xi_{1n}(r),\quad r\in(\hat R,
R)
\]
and the boundary condition $\xi_{1n}(R)=0$, we have 
\begin{eqnarray*}
&& (\lambda+2\mu)q''_{1n}(R)+(\lambda+\mu)\frac{1}{R}q'_{1n}(r)
-(\lambda+\mu)\frac{n^2}{R^2}q_{1n}(R)-\overline{M_{11}^{(n)}}q'_{1n}(R)
-\frac{{\rm i}n}{R}\overline{M_{21}^{(n)}}q_{1n}(R)\\
&=&(\lambda+2\mu)\left[-\frac{1}{\lambda+2\mu}\xi_{1n}(R)-\frac{1}{R}
q'_{1n}(R)-\left(\kappa_1^2
-\left(\frac{n}{R}\right)^2\right)q_{1n}(R)\right]\\
&&\qquad
+\left(\frac{\lambda+2\mu}{R}-\omega^2\frac{\overline{\alpha_{2n}}}{
\overline {\Lambda_n(R)}}\right)q'_{1n}(R)
+\left(-(\lambda+2\mu)\left(\frac{n}{R}\right)^2
+\omega^2\frac{1}{\overline{\Lambda_n(R)}}\left(\frac{n}{R}\right)^2\right)
q_{1n}(R)\\
&=& -\xi_{1n}(R)-\omega^2\frac{\overline{\alpha_{2n}}}{\overline{
\Lambda_n(R)} }q'_{1n}(R)+\left(-\omega^2
+\omega^2\frac{1}{\overline{\Lambda_n(R)}}\left(\frac{n}{R}\right)^2\right)
q_{1n}(R)\\
&=& -\xi_{1n}(R)-\frac{\omega^2}{\overline{\Lambda_n(R)}}\left(\overline{
\alpha_{2n}}q'_{1n}(R) +\left(\frac{n}{R}\right)^2 q_{1n}(R)
-\overline{\alpha_{1n}\alpha_{2n}} q_{1n}(R)-\left(\frac{n}{R}
\right)^2 q_{1n}(R)\right)\\
&=& -\xi_{1n}(R)
-\frac{\omega^2}{\overline{\Lambda_n(R)}}\overline{\alpha_{2n}} \left[
q'_{1n}(R) -\overline{\alpha_{1n}} q_{1n}(R)\right]=0.
\end{eqnarray*}
Hence we prove that $\mathscr B\boldsymbol{p}=\mathscr{T}^*\boldsymbol{p}$ on
$\partial B_R$. 

Moreover, we get from the Helmholtz decomposition that 
\begin{eqnarray*}
&&
\mu\Delta\boldsymbol{p}+(\lambda+\mu)\nabla\nabla\cdot\boldsymbol{p}+\omega^2
\boldsymbol{p} \\
&=& \nabla\left((\lambda+2\mu)\Delta q_1+\omega^2 q_1\right)
+{\bf curl}\left(\mu\Delta q_2+\omega^2 q_2\right)\\
&=& -\nabla\xi_1-{\bf curl}\xi_2=-\boldsymbol{\xi},
\end{eqnarray*}
which completes the proof.
\end{proof}

Based on Lemma \ref{Lemma_Xi_System} and Lemma \ref{Thm_Xi}, we have the
asymptotic properties of the solution to the dual problem 
\eqref{Dual_p} for large $|n|$. 

\begin{theorem}\label{Solu_p}
Let $\boldsymbol{p}$ be the solution of \eqref{Dual_p} and amdit the Fourier
series expansion
\[
\boldsymbol{p}(r,\theta)=\sum\limits_{n\in Z}
\left(p_n^r(r)\boldsymbol{e}_r
+p_n^{\theta}(r)\boldsymbol{e}_{\theta}\right)e^{{\rm
i}\,n\theta}.
\]
For sufficient large $|n|$, the Fourier coefficients $p_n^r, p_n^{\theta}$
satisfy the estimate
\begin{eqnarray*}
|p_n^r(R)|^2+|p_n^{\theta}(R)|^2 &\lesssim &
n^2\bigg(\frac{\hat R}{R}\bigg)^{2|n|+2}\left(|p_n^r(\hat R)|^2+|p_n^{\theta}
(\hat R)|^2\right)\\
&&+\frac{1}{|n|^2}\left(\|\xi_n^r\|_{L^{\infty}([\hat R,
R])}^2+\|\xi_n^{\theta}\|_{L^{\infty}([\hat R, R])}^2\right),
\end{eqnarray*}
where $\xi_n^r, \xi_n^\theta$ are the Fourier coefficients of
$\boldsymbol\xi$ in the polar coordinates and are given in \eqref{xi_pfc}. 
\end{theorem}

\begin{proof}
It follows from straighforward calculations that the second order systems \eqref{Dual_phi} have a unique solution, which is given by 
\begin{eqnarray}
q_{1n}(r) &=& \beta_{1n}(r)q_{1n}(\hat R)+\frac{{\rm
i}\pi}{4}\int_{\hat R}^{r}tW_{1n}(r,t)
\xi_{1n}(t){\rm d}t \notag\\
&&\qquad +\frac{{\rm i}\pi}{4}\int_{\hat R}^{R}t
\beta_{1n}(t)W_{1n}(\hat R, r)\xi_{1n}(t){\rm d}t \label{Dual_Phi},\\
q_{2n}(r) &=& \beta_{2n}(r)q_{2n}(\hat R)+\frac{{\rm
i}\pi}{4}\int_{\hat R}^{r}t W_{2n}(r,t)
\xi_{2n}(t){\rm d}t \notag\\
&&\qquad+\frac{{\rm i}\pi}{4}\int_{\hat R}^{R}t
\beta_{2n}(t)W_{2n}(\hat R, r)\xi_{2n}(t){\rm d}t \label{Dual_Psi},
\end{eqnarray}		
where
\[
\beta_{jn}(r)=\frac{H_n^{(2)}(\kappa_j r)}{H_n^{(2)}(\kappa_j
\hat R)},\quad W_{jn}(r,t)=H_n^{(1)}(\kappa_j r)H_n^{(2)}(\kappa_j
t)-H_n^{(1)}(\kappa_j t)H_n^{(2)}(\kappa_j r).
\]
Taking the derivative of \eqref{Dual_Phi}--\eqref{Dual_Psi} respective to
$r$ gives
\begin{eqnarray}\label{Dual_Dphi}
q'_{1n}(r) &=& \beta'_{1n}(r)q_{1n}(\hat R)+\frac{{\rm
i}\pi}{4}\int_{\hat R}^{r}t\partial_r W_{1n}(r,t)\xi_{1n}(t){\rm d}t
\notag\\
&&\qquad+\frac{{\rm i}\pi}{4}\int_{\hat R}^{R} t
\beta_{1n}(t)\partial_t W_{1n}(\hat R, r)\xi_{1n}(t){\rm d}t,
\end{eqnarray}
\begin{eqnarray}\label{Dual_Dpsi}
q'_{2n}(r) &=& \beta'_{2n}(r)q_{2n}(\hat R)+\frac{{\rm
i}\pi}{4}\int_{\hat R}^{r}t\partial_r
W_{2n}(r,t)\xi_{2n}(t){\rm d}t\notag\\
&&\qquad+\frac{{\rm i}\pi}{4}\int_{\hat R}^{R} t
\beta_{2n}(t)\partial_t W_{2n}(\hat R, r)\xi_{2n}(t){\rm d}t.
\end{eqnarray}
Evaluating \eqref{Dual_Phi}--\eqref{Dual_Psi} and
\eqref{Dual_Dphi}--\eqref{Dual_Dpsi} at $r=R$ and $r=\hat R$, respectively, we
may verify that 
\begin{eqnarray*}
q_{1n}(R) &=& \beta_{1n}(R)q_{1n}(\hat R)+\frac{{\rm
i}\pi}{4}\int_{\hat R}^{R}t \beta_{1n}(R)  W_{1n}(\hat R,
t)\xi_{1n}(t){\rm d}t,\\
q_{2n}(R) &=& \beta_{2n}(R)q_{2n}(\hat R)+\frac{{\rm
i}\pi}{4}\int_{\hat R}^{R}t \beta_{2n}(R)  W_{2n}(\hat R,
t)\xi_{2n}(t){\rm d}t,\\
q'_{1n}(\hat R) &=&
\beta'_{1n}(\hat R) q_{1n}(\hat R)+\frac{1}{\hat R}\int_{\hat R}^{R} t
\beta_{1n}(t) \xi_{1n}(t){\rm d}t,\\
q'_{2n}(\hat R) &=&
\beta'_{2n}(\hat R) q_{2n}(\hat R)+\frac{1}{\hat R}\int_{\hat R}^{R} t
\beta_{2n}(t) \xi_{2n}(t){\rm d}t.
\end{eqnarray*}

It follows from the Helmholtz decomposition that  
\begin{equation}\label{Solu_p-s1}
p_n^{r}(r)=q'_{1n}(r)+\frac{{\rm i}n}{r}q_{2n}(r), \quad
p_n^{\theta}(r)=\frac{{\rm i}n}{r}q_{1n}(r)-q'_{2n}(r).
\end{equation}	
Evaluating \eqref{Solu_p-s1} at $r=R$, noting
$\beta'_{jn}(R)=\overline{\alpha_{jn}(R)}$ and
$q'_{jn}(R)=\overline{\alpha_{jn}(R)} q_{jn}(R)$, we obtain 
\begin{equation}\label{Solu_p-s2}
\begin{bmatrix}
p_n^{r}(R) \\[2pt]
p_n^{\theta}(R)
\end{bmatrix}=U_n(R)
\begin{bmatrix}
q_{1n}(\hat R) \\ 
q_{2n}(\hat R)
\end{bmatrix}+\frac{{\rm i}\pi}{4}U_n(R)
\begin{bmatrix}
\int_{\hat R}^{R} t W_{1n}(\hat R, t)\xi_{1n}(t){\rm d}t\\[5pt]
\int_{\hat R}^{R} t W_{2n}(\hat R, t)\xi_{2n}(t){\rm d}t
\end{bmatrix},
\end{equation}
where 
\[
U_n(R)=\begin{bmatrix}
\overline{\alpha_{1n}(R)} & \frac{{\rm i}n}{R}\\[5pt]
\frac{{\rm i}n}{R} & -\overline{\alpha_{2n}(R)}
\end{bmatrix}
\begin{bmatrix}
\beta_{1n}(R) & 0 \\[2pt]
0 & \beta_{2n}(R)
\end{bmatrix}.
\]

Similarly, evaluating \eqref{Solu_p-s1} at $r=\hat R$ and noting
$\beta'_{jn}(\hat R)=\overline{\alpha_{jn}(\hat R)}$ yield that 
\begin{equation}\label{Solu_p-s3}
\begin{bmatrix}
p_n^{r}(\hat R) \\[2pt]
p_n^{\theta}(\hat R)
\end{bmatrix}=K_n(\hat R)
\begin{bmatrix}
q_{1n}(\hat R) \\ 
q_{2n}(\hat R)
\end{bmatrix}+
\begin{bmatrix}
\eta_1\\
\eta_2
\end{bmatrix},
\end{equation}
where 
\[
K_n(\hat R)=\begin{bmatrix}
\overline{\alpha_{1n}(\hat R)} & \frac{{\rm i}n}{\hat R}\\[5pt]
\frac{{\rm i}n}{\hat R} & -\overline{\alpha_{2n}(\hat R)} \end{bmatrix},
\]
and
\[
\eta_{1n}=\frac{1}{\hat R}\int_{\hat R}^{R} t
\beta_{1n}(t)\xi_{1n}(t){\rm d}t,\quad 
\eta_{2n}=-\frac{1}{\hat R}\int_{\hat R}^{R} t
\beta_{2n}(t)\xi_{2n}(t){\rm d}t.
\]
Solving \eqref{Solu_p-s3} for $q_{1n}(\hat R), q_{2n}(\hat R)$ in terms of
$p_n^r(\hat R), p_n^{r}(\hat R)$ gives
\begin{equation}\label{Solu_p-s4}
\begin{bmatrix}
q_{1n}(\hat R) \\ 
q_{2n}(\hat R)
\end{bmatrix}=\frac{V_n(\hat R)}{\overline{\Lambda_n(\hat R)}}
\begin{bmatrix}
p_n^{r}(\hat R)-\eta_{1n}\\ 
p_n^{\theta}(\hat R)-\eta_{2n}
\end{bmatrix},
\end{equation} 
where
\[
 \Lambda_n(\hat R)=\left(\frac{n} {\hat R}\right)^2-\alpha_{1n}(\hat
R)\alpha_{2n}(\hat R),\quad V_n(\hat R)=\begin{bmatrix}
-\overline{\alpha_{2n}(\hat R)} & -\frac{{\rm i}n}{\hat R}\\[5pt]
-\frac{{\rm i}n}{\hat R} & \overline{\alpha_{1n}(\hat R)} \end{bmatrix}.
\]
Substituting \eqref{Solu_p-s4} into \eqref{Solu_p-s2} yields\\
\begin{equation}\label{pp_pp}
\begin{bmatrix}
p_n^{r}(R) \\[2pt]
p_n^{\theta}(R)
\end{bmatrix}=\frac{U_n(R) V_n(\hat R)}{\overline{\Lambda_n(\hat R)}}
\begin{bmatrix}
p_n^{r}(\hat{R}) \\[2pt]
p_n^{\theta}(\hat{R})
\end{bmatrix}+\frac{{\rm i}\pi}{4}U_n(R)
\begin{bmatrix}
\int_{\hat R}^{R} t W_{1n}(\hat R, t) \xi_{1n}(t){\rm d}t\\[2pt]
\int_{\hat R}^{R} t W_{2n}(\hat R, t) \xi_{2n}(t){\rm d}t
\end{bmatrix}-\frac{U_n(R)V_n(\hat R)}{\overline{\Lambda_n(\hat R)}}
\begin{bmatrix}
\eta_{1n} \\
\eta_{2n}
\end{bmatrix}.
\end{equation}
\\
Following proofs in Lemmas \ref{lemma_xi_5} and \ref{Lamn}, we may similarly
show that for sufficiently large $|n|$
\[
\left|\frac{U_n(R) V_n(\hat R)}{\overline{\Lambda_n(\hat R)}}\right|\lesssim
|n|\bigg(\frac{\hat R}{R}\bigg)^{|n|}.
\]\\
For fixed $t$ and sufficiently large $|n|$, using \eqref{SoluXi_1} and
\eqref{SoluXi_2}, we may easily show 
\begin{eqnarray}
|\xi_{1n}(t)| &\lesssim& \left(\|\xi_n^r\|_{L^{\infty}([\hat R,
R])}+\|\xi_n^{\theta}\|_{L^{\infty}([\hat R, R])}\right)
\int_{t}^{R}\left(\frac{r}{t}\right)^{|n|}{\rm d}r,
\label{AsymXi_1}\\
|\xi_{2n}(t)|&\lesssim& \left(\|\xi_n^r\|_{L^{\infty}([\hat R,
R])}+\|\xi_n^{\theta}\|_{L^{\infty}([\hat R, R])}\right)
\int_{t}^{R}\left(\frac{r}{t}\right)^{|n|}{\rm d}r \label{AsymXi_2}.
\end{eqnarray}		
By \eqref{AsymXi_1}--\eqref{AsymXi_2} and 
\[
W_{jn}(\hat R, t) \sim
-\frac{2{\rm i}}{\pi |n|}\left[\bigg(\frac{t}{\hat R}\bigg)^{|n|}
-\bigg(\frac{\hat R}{t}\bigg)^{|n|}\right],\quad 
\beta_{jn}(t) \sim \bigg(\frac{\hat R}{t}\bigg)^{|n|},
\]
we get
\begin{eqnarray*}
\left|\int_{\hat R}^{R} t W_{jn}(\hat R,
t)\xi_{jn}(t){\rm d}t\right| &\lesssim&
\left(\|\xi_n^r\|_{L^{\infty}([\hat R,
R])}+\|\xi_n^{\theta}\|_{L^{\infty}([\hat R, R])}\right)\frac{1}{|n|}
\int_{\hat
R}^{R}t\left(\frac{t}{\hat R}\right)^{|n|}\int_{t}^{R}\left(\frac{r}{t}
\right)^{|n|}{\rm d}r{\rm d}t\\
&=&
\left(\|\xi_n^r\|_{L^{\infty}([\hat R,
R])}+\|\xi_n^{\theta}\|_{L^{\infty}([\hat R, R])}\right)\frac{1}{|n|(|n|+1)}\\
&&\qquad\times\int_{\hat
R}^{R}t\left(\frac{1}{\hat R}\right)^{|n|}\left(R^{|n|+1}-t^{|n|+1}\right){\rm d}t\\
&\lesssim & \left(\|\xi_n^r\|_{L^{\infty}([\hat R,
R])}+\|\xi_n^{\theta}\|_{L^{\infty}([\hat R, R])}\right)
\frac{1}{|n|^2}\bigg(\frac{R}{\hat R}\bigg)^{|n|}
\end{eqnarray*}	
and
\begin{eqnarray*}
\left|\frac{1}{\hat R}\int_{\hat R}^{R} t
\beta_{jn}(t)\xi_{jn}(t){\rm d}t\right| &\lesssim&
\left(\|\xi_n^r\|_{L^{\infty}([\hat R,
R])}+\|\xi_n^{\theta}\|_{L^{\infty}([\hat R, R])}\right)
\int_{\hat R}^{R} t
\bigg(\frac{\hat R}{t}\bigg)^{|n|}\int_{t}^{R}\left(\frac{r}{t}\right)^{|n|}
{\rm d}r{\rm d}t\\
&=&
\left(\|\xi_n^r\|_{L^{\infty}([\hat R,
R])}+\|\xi_n^{\theta}\|_{L^{\infty}([\hat R, R])}\right)\frac{1}{1+|n|}\\
&&\qquad \times
\int_{\hat R}^{R} 
\bigg(\frac{\hat R}{t}\bigg)^{|n|}\frac{R^{|n|+1}-t^{|n|+1}}{t^{|n|-1}}
{\rm d}t\\
&\lesssim &\left(\|\xi_n^r\|_{L^{\infty}([\hat R,
R])}+\|\xi_n^{\theta}\|_{L^{\infty}([\hat R, R])}\right)
\frac{1}{|n|^2}\bigg(\frac{R}{\hat R}\bigg)^{|n|}.
\end{eqnarray*}
Substituting the above estimates into \eqref{pp_pp}, we obtain 
\begin{eqnarray*}
|p_n^r(R)|^2+|p_n^{\theta}(R)|^2 &\lesssim &
n^2\bigg(\frac{\hat R}{R}\bigg)^{2|n|+2}\left(|p_n^r(\hat R)|^2+|p_n^{\theta}
(\hat R)|^2\right)\\
&&+\frac{1}{|n|^2}\left(\|\xi_n^r\|_{L^{\infty}([\hat R,
R])}^2+\|\xi_n^{\theta}\|_{L^{\infty}([\hat R, R])}^2\right),
\end{eqnarray*}
which completes the proof. 
\end{proof}

Using Theorem \ref{Solu_p}, we may estimate the last term in
\eqref{LemmaXi_3a}.

\begin{lemma}
 Let $\boldsymbol p$ be the solution of the dual problem \eqref{Dual_p}. For
sufficiently large $N$, the following estimate holds
 \[
\left|\int_{\partial B_R}\left(\mathscr{T}-\mathscr{T}_N\right)\boldsymbol{\xi}
\cdot\overline{ \boldsymbol{p}}{\rm d}s\right|\lesssim
\frac{1}{N}\|\boldsymbol{\xi}\|^2_{\boldsymbol{H}^{1}(\Omega)}. 
 \]
\end{lemma}

\begin{proof}
Using the definitions of the DtN operators $\mathscr T$ and  $\mathscr T_N$ and
Lemma \ref{Poincare}, we have 
\begin{eqnarray}
&& \left|\int_{\partial
B_R}\left(\mathscr{T}-\mathscr{T}_N\right)\boldsymbol{\xi}\cdot\overline{
\boldsymbol{p}}{\rm d}s \right|
\leq 2\pi R\sum\limits_{|n|>N} \left|\left(M_n
\boldsymbol{\xi}_n(R)\right)\cdot \overline{\boldsymbol{p}}_n(R)\right| \notag\\
&&\lesssim 2\pi R \sum\limits_{|n|>N} |n|
\left(|\xi_n^r(R)|+|\xi_n^{\theta}(R)|\right)
\left(|p_n^r(R)|+|p_n^{\theta}(R)|\right)	\notag \\
&& \lesssim \sum\limits_{|n|>N}\left((1+n^2)^{1/2}|n|\right)^{-1/2}
\left[\sum\limits_{|n|>N}(1+n^2)^{1/2}\left(|\xi_n^r(R)|+|\xi_n^{\theta}
(R)|\right)^2\right]^{1/2}	\notag\\
&& \qquad	\times\left[\sum\limits_{|n|>N} |n|^3
\left(|p_n^r(R)|+|p_n^{\theta}(R)|\right)^2\right]^{1/2}\notag \\
&& \lesssim N^{-1} \|\boldsymbol{\xi}\|_{H^{1/2}(\partial B_R)}	
\left[\sum\limits_{|n|>N} |n|^3
\left(|p_n^r(R)|^2+|p_n^{\theta}(R)|^2\right)\right]^{1/2}\notag\\
&& \lesssim N^{-1} \|\boldsymbol{\xi}\|_{H^{1}(\Omega)}	
\left[\sum\limits_{|n|>N} |n|^3
\left(|p_n^r(R)|^2+|p_n^{\theta}(R)|^2\right)\right]^{1/2}\label{P_2}.
\end{eqnarray}

Following \cite{JLLZ-jsc17}, we let $t\in [\hat R, R]$ and assume, without loss
of generality, that $t$ is closer to the left endpoint $\hat R$
than the right endpoint $R$. Denote $\zeta=R-\hat R$. Then we have $R-t\geq
\frac{\zeta}{2}$. Thus
\begin{eqnarray*}
|\xi_n^{(r, \theta)}(t)|^2 &=&
\frac{1}{R-t}\int_{R}^{t}\left((R-s)|\xi_n^{(r, \theta)}(s)|^2\right)' {\rm
d}s\\
&=& \frac{1}{R-t}\int_{R}^{t}\left(-|\xi_n^{(r, \theta)}(s)|^2+2\left(R-s\right)
\Re\big(\xi_n^{(r, \theta)\prime} (s)
\overline{\xi_n^{(r, \theta)}(s)}\big)\right){\rm d}s \\
&\leq&  \frac{1}{R-t}\int_{t}^{R}|\xi_n^{(r, \theta)}(s)|^2{\rm d}s
+2\int_{\hat R}^{R}  |\xi_n^{(r, \theta)}(s)||\xi_n^{(r, \theta)\prime}(s)|{\rm
d}s,
\end{eqnarray*}
which implies that
\begin{eqnarray*}
\|\xi_n^{(r, \theta)}\|^2_{L^{\infty}([\hat R, R])} &\leq&
\frac{2}{\zeta}\|\xi_n^{(r, \theta)}\|^2_{L^{2}([\hat R, R])}
+2\|\xi^{(r, \theta)}_{n}\|_{L^2([\hat R,
R])}\|\xi^{(r, \theta)\prime}_{n}\|_{L^2([\hat R, R])}\\
&\leq& \left(\frac{2}{\zeta}+|n|\right)\|\xi_n^{(r, \theta)}\|^2_{L^2([\hat R,
R])} +|n|^{-1}\|\xi_n^{(r, \theta)\prime}\|^2_{L^2([\hat R, R])}.
\end{eqnarray*}

Using Lemma \ref{Solu_p} and the Cauchy--Schwarz inequality, we get 
\begin{eqnarray*}
&& \sum\limits_{|n|>N} |n|^3
\left(|p_n^r(R)|^2+|p_n^{\theta}(R)|^2\right)\\
&& \lesssim \sum\limits_{|n|>N} |n|^3
\left\{n^2\bigg(\frac{\hat
R}{R}\bigg)^{2|n|+2}\left(|p_n^r(\hat R)|^2+|p_n^{\theta} (\hat R)|^2\right)
+\frac{1}{|n|^2}\left(\|\xi_n^r\|_{L^{\infty}([\hat R,
R])}^2+\|\xi_n^{\theta}\|_{L^{\infty}([\hat R, R])}^2\right)\right\}\\
&&\lesssim\sum\limits_{|n|>N} |n|^5
\bigg(\frac{\hat R}{R}\bigg)^{|2n|}\left(|p_n^r(\hat R)|^2+|p_n^{\theta}
(\hat R)|^2\right)
+\sum\limits_{|n|>N} |n|\left(\|\xi_n^{r}\|_{L^{\infty}([\hat R,
R])}^2+\|\xi_n^{\theta}\|_{L^{\infty}([\hat R, R])}^2 \right)\\
&& := I_1+I_2.	
\end{eqnarray*}
Noting that the function $t^{4} e^{-2t}$ is bounded on $(0, +\infty)$, we have
\[
I_1\lesssim
\max\limits_{|n|>N}\left(n^{4}\bigg(\frac{\hat R}{R}\bigg)^{|2n|}\right)
\sum\limits_{|n|>N}
|n|\left(|p_n^r(\hat R)|^2+|p_n^{\theta}(\hat R)|^2\right)
\lesssim \|\boldsymbol{p}\|_{H^{1/2}(\partial B_{\hat R})}^2
\lesssim \|\boldsymbol{\xi}\|_{H^{1}(\Omega)}^2,
\]
where the last inequality uses the stability of the dual problem \eqref{Dual_p}.
For $I_2$, we can show that
\begin{eqnarray*}
I_2 &\lesssim& \sum\limits_{|n|>N}
\left[|n|\left(\frac{2}{\zeta}+|n|\right)\left(\|\xi_n^r\|_{L^{2}([\hat R,
R])}^2 +\|\xi_n^{\theta}\|_{L^{2}([\hat R, R])}^2 \right)+
\left(\|\xi_n^{r\prime}\|_{L^{2}([\hat R, R])}^2
+\|\xi_n^{\theta\prime}\|_{L^{2}([\hat R, R])}^2
\right)\right]\\
&\leq& \sum\limits_{|n|>N}
\left[\left(\frac{2}{\zeta}|n|+n^2\right)\|\boldsymbol{\xi}_n\|_{L^{2}([\hat
R , R])}^2 + \|\boldsymbol{\xi}_n^\prime\|_{L^{2}([\hat R, R])}^2\right].
\end{eqnarray*}

On the other hand, a simple calculation yields
\begin{eqnarray*}
\|\xi_n^{(r, \theta)}\|^2_{H^{1}(B_R\setminus B_{\hat R})} &=& 2\pi
\sum\limits_{n\in\mathbb{Z}} \int_{\hat R}^{R}
\left[\left(r+\frac{n^2}{r}\right)|\xi_n^{(r,
\theta)}(r)|^2+r|\xi_n^{(r, \theta)\prime} (r)|^2\right]{\rm d}r \\
&\geq & 2\pi \sum\limits_{n\in\mathbb{Z}}\int_{\hat R}^{R} 
\left[\left(\hat R+\frac{n^2}{R}\right)|\xi_n^{(r, \theta)}(r)|^2
+\hat R |\xi_n^{(r, \theta)\prime}(r)|^2\right]{\rm d}r.
\end{eqnarray*}
It is easy to note that
\[
\frac{2}{\zeta}|n|+n^2\lesssim \hat R+\frac{n^2}{R}.
\]
Combining the above estimates, we obtain 
\[
I_2\lesssim  \|\boldsymbol{\xi}\|^2_{H^{1}(B_R\setminus B_{\hat{R}})}\leq
\|\boldsymbol{\xi}\|^2_{H^{1}(\Omega)},
\]
which gives 
\[
\sum\limits_{|n|>N}|n|^3
\left(|p_n^{r}(R)|+|p_n^{\theta}(R)|\right)^2\lesssim
\|\boldsymbol{\xi}\|^2_{H^{1}(\Omega)}.
\]
Substituting the above inequality into \eqref{P_2}, we get
\begin{equation}\label{4thTerm}
\left|\int_{\partial
B_R}\left(\mathscr{T}-\mathscr{T}_N\right)\boldsymbol{\xi}\cdot\overline{
\boldsymbol{p}}{\rm d}s \right|\lesssim
\frac{1}{N}\|\boldsymbol{\xi}\|^2_{\boldsymbol{H}^{1}(\Omega)},
\end{equation}
which completes the proof. 
\end{proof}

Now, we prove the main result of this paper.

\begin{proof}
By Lemma \ref{LemmaXi_1}, Lemma \ref{lemma_xi_3}, and 
Lemma \ref{lemma_xi_6},  we obtain 
\begin{eqnarray*}
\vvvert\boldsymbol{\xi}\vvvert^2_{\boldsymbol{H}^{1}(\Omega)} &=& \Re
b(\boldsymbol{\xi}, \boldsymbol{\xi})+\Re\int_{\partial B_R}
\left(\mathscr{T}-\mathscr{T}_N\right)\boldsymbol{\xi}\cdot\overline{\boldsymbol
{\xi}}{\rm d}s +2\omega^2 \int_{\Omega}
\boldsymbol{\xi}\cdot\overline{\boldsymbol{\xi}}{\rm d}\boldsymbol{x}
+\Re\int_{\partial B_R}
\mathscr{T}_N\boldsymbol{\xi}\cdot\overline{\boldsymbol{\xi}}{\rm d}s\\	
&\leq& C_1\left[\left(\sum\limits_{K\in\mathcal M_h} \eta^2_K\right)^{1/2}
+\max_{|n|>N}|n|\bigg(\frac{\hat R}{R}\bigg)^{|n|}
\|\boldsymbol{u}^{\rm inc}\|_{\boldsymbol{H}^{1}(\Omega)}\right] 	
\|\boldsymbol{\xi}\|_{\boldsymbol H^{1}(\Omega)}\\
&&\qquad +\left(C_2+C(\delta)\right)\|\boldsymbol{\xi}\|^2_{\boldsymbol{L}
^2(\Omega)} +\bigg(\frac{R}{
\hat R}\bigg)\delta\|\boldsymbol{\xi}\|^2_{\boldsymbol H^{1}(\Omega)}.	
\end{eqnarray*}		
Using \eqref{Equal_norm} and choosing $\delta$ such that
$\frac{R}{\hat R}\frac{\delta}{\min(\mu, \omega^2)}<\frac{1}{2}$, we get
\begin{eqnarray}\label{Estimate_H1}
\vvvert\boldsymbol{\xi}\vvvert^2_{\boldsymbol{H}^{1}(\Omega)} &\leq& 
2C_1\left[\left(\sum\limits_{K\in\mathcal M_h} \eta^2_{K}\right)^{1/2}
+\max_{|n|>N}|n|\bigg(\frac{\hat R}{R}\bigg)^{|n|}\|\boldsymbol{u}^{\rm
inc}\|_{\boldsymbol{H}^{1}(\Omega)}\right] 
\|\boldsymbol{\xi}\|_{\boldsymbol H^{1}(\Omega)}
\notag\\
&& +2\left(C_2+C(\delta)\right)\|\boldsymbol{
\xi}\|^2_{\boldsymbol{L} ^2(\Omega)}.
\end{eqnarray}
It follows from \eqref{LemmaXi_3a}, \eqref{4thTerm}, and \eqref{Equal_norm}
that we have 
\begin{eqnarray*}
\|\boldsymbol{\xi}\|^{2}_{\boldsymbol{L}^2(\Omega)}
&=& b(\boldsymbol{\xi}, \boldsymbol{p})
+\int_{\partial
B_R}\left(\mathscr{T}-\mathscr{T}_N\right)\boldsymbol{\xi}\cdot\overline{
\boldsymbol{p}}{\rm d}s
-\int_{\partial
B_R}\left(\mathscr{T}-\mathscr{T}_N\right)\boldsymbol{\xi}\cdot\overline{
\boldsymbol{p}}{\rm d}s\\
&\lesssim& \left[\left(\sum\limits_{K\in\mathcal M_h}
\eta^2_{K}\right)^{1/2}
+\max_{|n|>N}|n|\bigg(\frac{\hat R}{R}\bigg)^{|n|}
\|\boldsymbol{u}^{\rm inc}\|_{\boldsymbol{H}^{1}(\Omega)}\right] 	
\|\boldsymbol{\xi}\|_{H^{1}(\Omega)}+\frac{1}{N}
\|\boldsymbol{\xi}\|^2_{H^{1}(\Omega)}.
\end{eqnarray*}
Substituting the above estimate intro \eqref{Estimate_H1} and taking
sufficiently large $N$ such
that 
\[
\frac{2\left(C_2+C(\delta)\right)}{N}\frac{1}{\min(\mu, \omega^2)}<1,
\]
we obtain 
\begin{eqnarray*}
\vvvert\boldsymbol{u}-\boldsymbol{u}_N^h
\vvvert^2_{\boldsymbol{H}^{1}(\Omega)}\lesssim 
\left(\sum\limits_{K\in\mathcal M_h} \eta^2_{K}\right)^{1/2}
+\max_{|n|>N}\left(|n|\bigg(\frac{\hat R}{R}\bigg)^{|n|}\right)
\|\boldsymbol{u}^{\rm inc}\|_{\boldsymbol{H}^{1}(\Omega)}.
\end{eqnarray*}
which completes the proof of theorem.
\end{proof}

\section{Implementation and numerical experiments}\label{Section_ne}

In this section, we discuss the algorithmic implementation of the adaptive
finite element DtN method and present two numerical examples to demonstrate the
effectiveness of the proposed method. 

\begin{table}
\caption{The adaptive finite element DtN method for the elastic wave scattering
problem.}
\hrule \hrule
\vspace{0.8ex}
\begin{enumerate}		
\item Given the tolerance $\epsilon>0, \theta\in(0,1)$;
\item Fix the computational domain $\Omega=B_R\setminus \overline{D}$ by
choosing the radius $R$;
\item Choose $R^\prime$ and $N$ such that $\epsilon_N\leq 10^{-8}$;
\item Construct an initial triangulation $\mathcal M_h$ over $\Omega$ and
compute error estimators;
\item While $\epsilon_h>\epsilon$ do
\item \qquad Refine the mesh $\mathcal M_h$ according to the strategy:
\[
\text{if } \eta_{\hat{T}}>\theta \max\limits_{T\in \mathcal M_h}
\eta_{T}, \text{ then refine the element } \hat{T}\in M_h;
\]
\item \qquad Denote refined mesh still by $\mathcal M_h$, solve the discrete
problem \eqref{Variational_3} on the new mesh $\mathcal M_h$;
\item \qquad Compute the corresponding error estimators;
\item End while.
\end{enumerate}
\vspace{0.8ex}
\hrule\hrule
\end{table}

\subsection{Adaptive algorithm}

Based on the a posteriori error estimate from Theorem \ref{Main_Result},
we use the FreeFem \cite{H-jnm12} to implement the adaptive algorithm of
the linear finite element formulation. It is shown in Theorem \ref{Main_Result}
that the a posteriori error estimator consists two parts: the finite element
discretization error $\epsilon_h$ and the DtN truncation error $\epsilon_N$
which dependents on the truncation number $N$. Explicitly
\begin{equation}\label{epsilonN}
\epsilon_h = \left(\sum\limits_{T\in\mathcal M_h}
\eta^2_{T}\right)^{1/2}, \quad 
\epsilon_N =\max_{|n|\geq N}\left(|n|\bigg(\frac{\hat{R}}{R}\bigg)^{|n|}\right)
\|\boldsymbol{u}^{\rm inc}\|_{\boldsymbol{H}^{1}(\Omega)}. 
\end{equation}

In the implementation, we choose $\hat{R}, R$, and $N$ based on \eqref{epsilonN} to
make sure that the finite element discretization 
error is not polluted by the DtN truncation error, i.e., $\epsilon_N$ is
required to be very small compared to $\epsilon_h$, for example, 
$\epsilon_N\leq 10^{-8}$. For simplicity, in the following numerical
experiments, $\hat R$ is chosen such that the obstacle is exactly contained
in the disk $B_{\hat R}$, and $N$ is taken to be the smallest positive integer
such that $\epsilon_N\leq 10^{-8}$. The algorithm is shown in Table 1
for the adaptive finite element DtN method for solving the elastic wave
scattering problem. 

\subsection{Numerical experiments}

We report two examples to demonstrate the performance of the proposed method.
The first example is a disk and has an analytical solution; the second example
is a U-shaped obstacle which is commonly used to test numerical solutions 
for the wave scattering problems. In each example, we plot the magnitude of the
numerical solution to give an intuition where the mesh should be refined, and
also plot the actual mesh obtained by our algorithm to show the agreement. The a
posteriori error is plotted against the number of nodal points to show the
convergence rate. In the first example, we compare the numerical results by
using the uniform and adaptive meshes to illustrate the effectiveness of the
adaptive algorithm. 

{\em Example 1.} This example is constructed such that it has an exact solution.
Let the obstacle $D=B_{0.5}$ be a disk with radius 0.5 and take
$\Omega=B_{1}\setminus \overline{B}_{0.5}$, i.e., $\hat{R}=0.5, R=1$. If we
choose the incident wave as
\[
\boldsymbol{u}^{\rm inc}(\boldsymbol x) 
= - \frac{\kappa_1 H_0^{(1)'}(\kappa_1 r)}{r}\begin{pmatrix}
x\\y \end{pmatrix}
-\frac{\kappa_2 H_0^{(1)'}(\kappa_2
r)}{r}\begin{pmatrix} y\\-x \end{pmatrix},\quad r=(x^2+y^2)^{1/2},
\]
then it is easy to check that the exact solution is
\begin{equation*}\label{Example1_exact}
\boldsymbol{u}(\boldsymbol x) = \frac{\kappa_1 H_0^{(1)'}(\kappa_1
r)}{r}\begin{pmatrix} x\\y \end{pmatrix}
+\frac{\kappa_2 H_0^{(1)'}(\kappa_2
r)}{r}\begin{pmatrix} y\\-x \end{pmatrix},
\end{equation*}
where $\kappa_1$ and $\kappa_2$ are the compressional wave number and
shear wave number, respectively. 

In Table 2, numerical results are shown for the adaptive mesh
refinement and the
uniform mesh refinement, where $\text{DoF}_h$ stands for the degree of freedom
or the number of nodal points of the mesh $\mathcal M_h$, $\epsilon_h$ is the a
posteriori error estimate, and
$e_h=\|\boldsymbol{u}-\boldsymbol{u}_N^h\|_{\boldsymbol{H}^1(\Omega)}$ is the a
priori error. It can be seen that the adaptive mesh refinement requires fewer
$\text{DoF}_h$ than the uniform mesh refinement to reach the same level of
accuracy, which shows the advantage of using the adaptive mesh refinement.
Figure \ref{fig2} displays the curves of $\log e_h$ and $\log \epsilon_h$
versus $\log \text{DoF}_h$ for the uniform and adaptive mesh refinements with
$\omega=\pi, \lambda=2, \mu=1$, i.e., $\kappa_1=\pi/2, \kappa_2=\pi$. It
indicates that the meshes and the associated
numerical complexity are quasi-optimal,
i.e., $\|\boldsymbol{u}-\boldsymbol{u}_N^h\|_{\boldsymbol{H}^1(\Omega)}
=O\big(\text{DoF}_h^{-1/2}\big)$ holds asymptotically. Figure \ref{fig3}
plots the magnitude of the numerical solution and an adaptively refined mesh
with 15407 elements. We can see that the solution oscillates on the edge of the
obstacle but it is smooth away from the obstacle. This feature is caught by the
algorithm. The mesh is adaptively refined around the obstacle and is coarse
away from the obstacle.

\begin{table}
\begin{center}
\caption{Comparison of numerical results using adaptive mesh and uniform
mesh refinements for Example 1.}
\begin{tabular}{r|c|c|r|c|c}
\hline\hline
\multicolumn{3}{c|}{Adaptive mesh} &
\multicolumn{3}{c}{Uniform mesh}\\
\hline\hline
$\text{DoF}_h$ & $e_h$ & $\epsilon_h $ & $\text{DoF}_h$ & $e_h$ &
$\epsilon_h$\\
\hline
1745 & 0.4632 & 3.9693 &   1745 & 0.4632 & 3.9693 \\
\hline
2984 & 0.3256  &  2.6723 & 2667 & 0.3717 & 3.2365 \\
\hline
5559 & 0.2253 & 1.9293 & 5857 & 0.2494 & 2.0625\\
\hline
9030 & 0.1778 & 1.5054 & 10630 & 0.1851 & 1.5856\\
\hline
15407 & 0.1384 & 1.1686 & 20224 & 0.1330 & 1.1257\\
\hline
\end{tabular}
\end{center}
\end{table}

\begin{figure}
\centering
\includegraphics[width=0.5\textwidth]{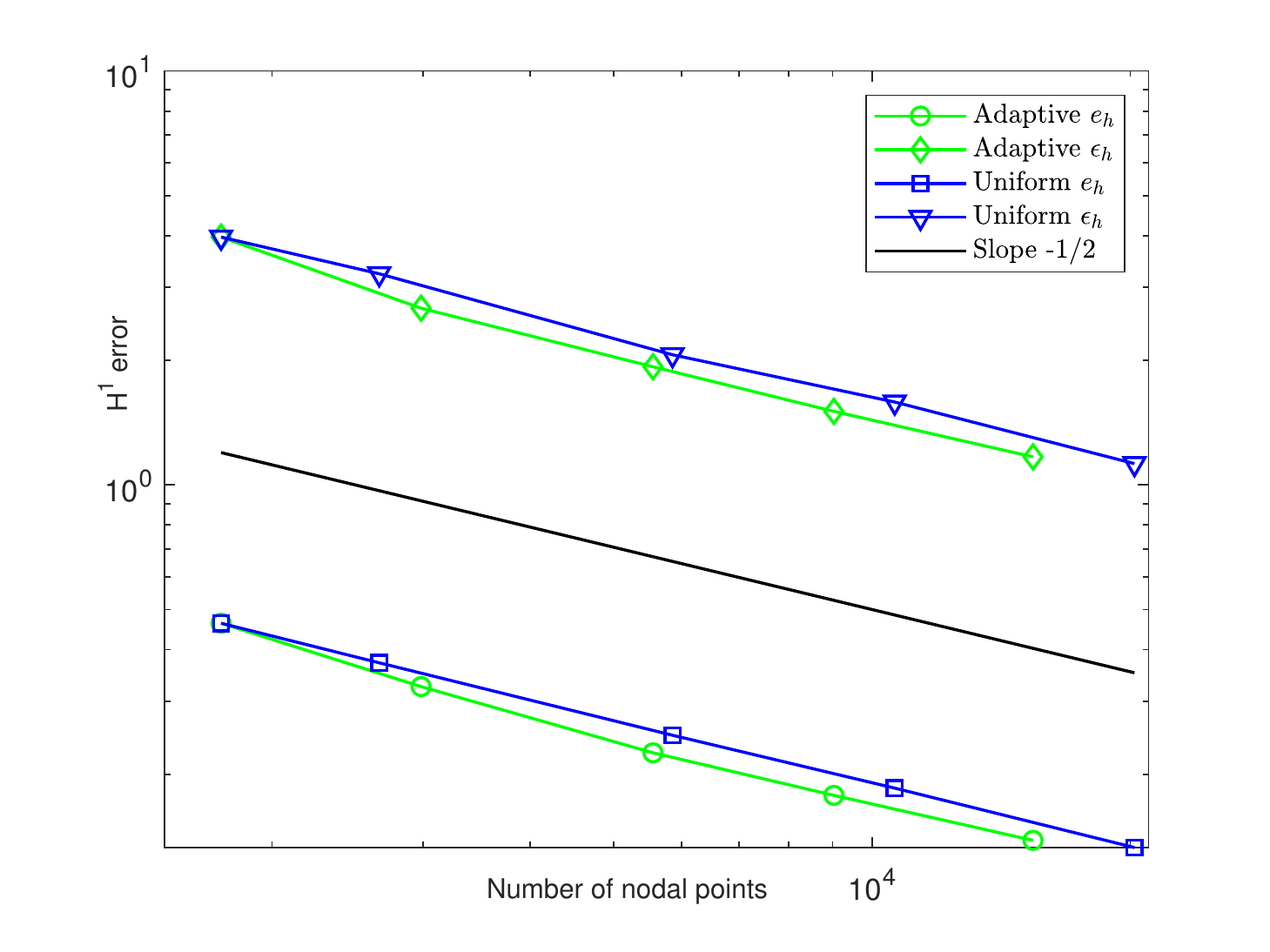}
\caption{Quasi-optimality of the a priori and a posteriori error
estimates for Example 1.}
\label{fig2}
\end{figure}

\begin{figure}
\centering
\includegraphics[width=0.36\textwidth]{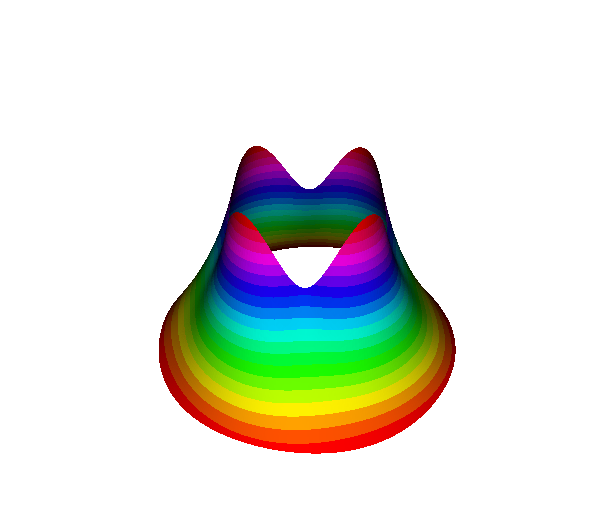}
\includegraphics[width=0.36\textwidth]{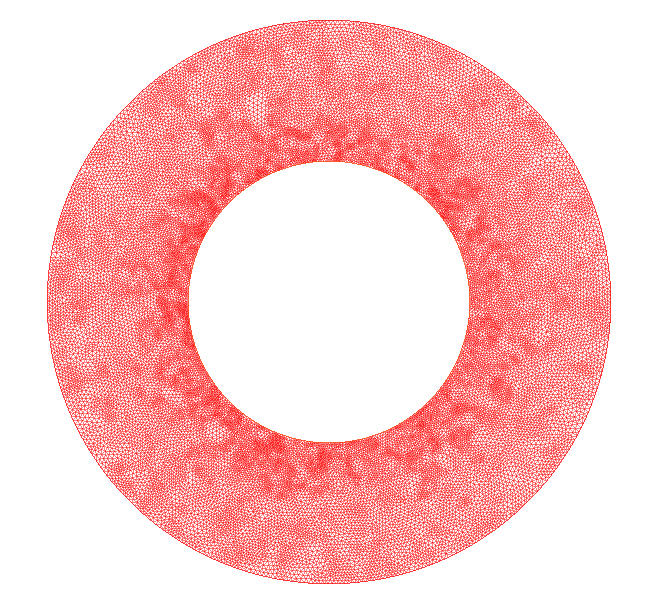}	
\caption{The numerical solution of Example 1. (left) the magnitude of the
numerical solution; (left) an adaptively refined mesh with 15407 elements.}
\label{fig3}
\end{figure}

{\em Example 2.} This example does not have an analytical solution. We consider
a compressional plane incident wave $\boldsymbol{u}^{\rm inc}(\boldsymbol x)
=\boldsymbol{d} e^{{\rm i}\,\kappa_1 \boldsymbol{x}\cdot \boldsymbol{d}}$ with
the incident direction $\boldsymbol{d}=(1,0)^\top$. The obstacle is U-shaped
and is contained in the rectangular domain $\left\{\boldsymbol x\in\mathbb R^2:
-2<x<2.2, -0.7<y<0.7\right\}$. Due to the problem geometry, the
solution contains singularity around the corners of the obstacle. We take $R=3, \hat R=2.31$. Figure \ref{fig4} shows the curve of $\log \epsilon_h$ versus
$\log \text{DoF}_h$ at different frequencies $\omega=1, \pi, 2\pi$. 
It demonstrates that the decay of the a posteriori error estimates are 
$O\big(\text{DoF}_h^{-1/2}\big)$. Figure \ref{fig5} plots the contour of the
magnitude of the numerical solution and its corresponding mesh by using the
parameters $\omega=\pi, \lambda=2, \mu=1$. Again, the algorithm does capture the
solution feature and adaptively refines the mesh around the corners of the
obstacle where the solution displays singularity. 

\begin{figure}
\centering
\includegraphics[width=0.5\textwidth]{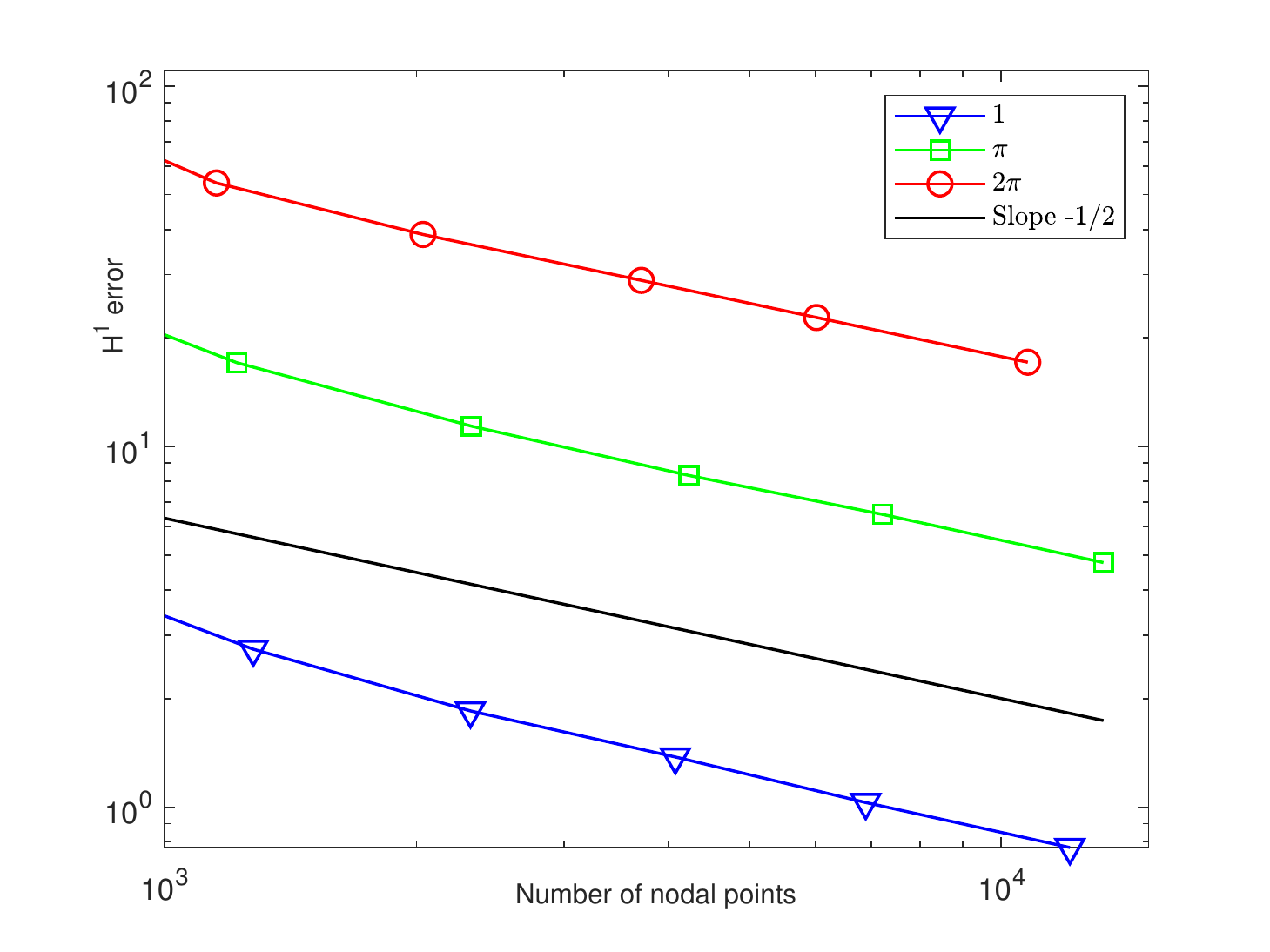}
\caption{Quasi-optimality of the a posteriori error estimates with
different frequencies for Example 2.  }
\label{fig4}
\end{figure}

\begin{figure}
\centering
\includegraphics[width=0.36\textwidth]{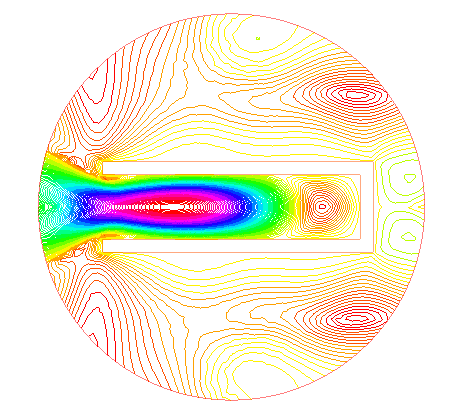}	
\includegraphics[width=0.36\textwidth]{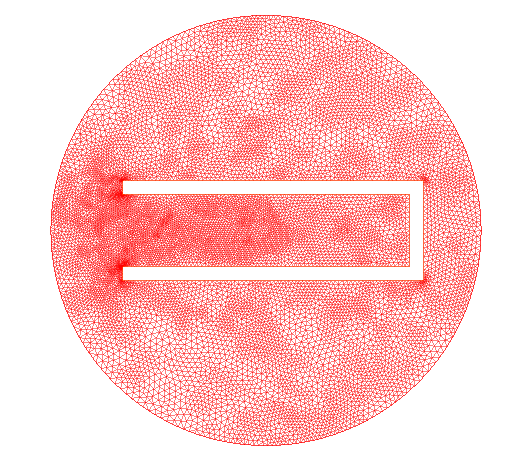}		
\caption{The numerical solution of Example 2. (left) The contour plot of the
magnitude of the solution; (right) an adaptively refined mesh with
12329 elements}
\label{fig5}
\end{figure}		

\section{Conclusion}\label{Section_c}

In this paper, we present an adaptive finite element DtN method for the elastic
obstacle scattering problem. Based on the Helmholtz decomposition, a
new duality argument is developed to obtain the a posteriori error estimate. It
not only takes into account of the finite element discretization error but also
includes the truncation error of the DtN operator. We show that the truncation
error decays exponentially with respect to the truncation parameter. The
posteriori error estimate for the solution of the discrete problem serves as a
basis for the adaptive finite element approximation. Numerical results show that
the proposed method is accurate and effective. This work provides a viable
alternative to the adaptive finite element PML method to solve the elastic
obstacle scattering problem. The method can be applied to solve many other wave
propagation problems where the transparent boundary conditions are used for
open domain truncation. Future work includes extending our analysis to the
adaptive finite element DtN method for solving the three-dimensional elastic
obstacle scattering problem, where a more complicated transparent boundary
condition needs to be considered. 

\appendix

\section{Transparent boundary conditions}\label{Section_ap}

In this section, we show the transparent boundary conditions for the scalar
potential functions $\phi, \psi$ and the displacement of the scattered field
$\boldsymbol u$ on $\partial B_R$. 

In the exterior domain $\mathbb R^2\setminus\overline B_R$, the solutions of
the Helmholtz equations \eqref{he} have the Fourier series expansions in the
polar coordinates: 
\begin{equation}\label{Fourier_Expan}
\phi(r,\theta)=\sum_{n\in\mathbb{Z}}\frac{H^{(1)}_n(\kappa_1 r)}{H_{n}^{(1)}
(\kappa_1 R)} \phi_n(R)e^{{\rm i}n\theta},\quad
\psi(r,\theta)=\sum_{n\in\mathbb{Z}}\frac{H^{(1)}_n(\kappa_2
r)}{H_{n}^{(1)}(\kappa_2 R)}\psi_n(R)e^{{\rm i}n\theta},
\end{equation} 
where $H_n^{(1)}$ is the Hankel function of the first kind with order $n$.
Taking the normal derivative of \eqref{Fourier_Expan}, we obtain the transparent
boundary condition for the scalar potentials $\phi, \psi$ on $\partial B_R$:
\begin{eqnarray}\label{ptbc}
\partial_r \phi=\mathscr{T}_1 \phi:=\sum\limits_{n\in\mathbb{Z}}
\frac{\kappa_1 H_n^{(1)'}(\kappa_1 R)}{H_{n}^{(1)}(\kappa_1 R)}
\phi_n(R)e^{{\rm i}n\theta},\quad
\partial_r \psi=\mathscr{T}_2
\psi:=\sum\limits_{n\in\mathbb{Z}} \frac{\kappa_2 H_n^{(1)'}(\kappa_2
R)}{H_{n}^{(1)}(\kappa_2 R)}
\psi_n(R)e^{{\rm i}n\theta}.		
\end{eqnarray}

The polar coordinates $(r, \theta)$ are related to the Cartesian coordinates
$\boldsymbol x=(x, y)$ by $x=r\cos\theta, y=r\sin\theta$ with the local
orthonormal basis $\{\boldsymbol e_r, \boldsymbol e_\theta\}$, where
$\boldsymbol e_r=(\cos\theta, \sin\theta)^\top, \boldsymbol
e_\theta=(-\sin\theta, \cos\theta)^\top$.

Define a boundary operator for the displacement of the scattered wave
\[
 \mathscr B\boldsymbol u=\mu\partial_r \boldsymbol
u+(\lambda+\mu)(\nabla\cdot\boldsymbol u) \boldsymbol
e_r\quad\text{on}\,\partial
B_R. 
\]
Based on the Helmholtz decomposition \eqref{he} and the transparent
boundary condition \eqref{ptbc}, it is shown in \cite{LWWZ-ip16} that the
scattered field $\boldsymbol u$ satisfies the transparent boundary condition 
\begin{equation}\label{utbc}
\mathscr B\boldsymbol u=(\mathscr T\boldsymbol u)(R, \theta):=\sum_{n\in\mathbb
Z}M_n\boldsymbol u_n(R) e^{{\rm i}n\theta}\quad\text{on}\,\partial B_R,
\end{equation}
where 
\[
 \boldsymbol u(R, \theta)=\sum_{n\in\mathbb Z}\boldsymbol u_n(R) e^{{\rm
i}n\theta}=\sum_{n\in\mathbb Z}\big(u_n^r(R)\boldsymbol e_r +
u_n^\theta(R)\boldsymbol e_\theta\big) e^{{\rm i}n\theta}
\]
and $M_n$ is a $2\times 2$ matrix defined by 
\begin{equation}\label{Mn}
 M_n=\begin{bmatrix}
      M_{11}^{(n)} & M_{12}^{(n)}\\
      M_{21}^{(n)} & M_{22}^{(n)}
     \end{bmatrix}
=\frac{1}{\Lambda_n(R)}\begin{bmatrix}
                         N_{11}^{(n)} & N_{12}^{(n)}\\
                         N_{21}^{(n)} & N_{22}^{(n)}
                        \end{bmatrix}.
\end{equation}
Here 
\begin{equation}\label{alpha_jn}
 \Lambda_n(R)=\left(\frac{n}{R}\right)^2-\alpha_{1n}(R)\alpha_{2n}(R),
\quad  \alpha_{jn}(R)=\frac{\kappa_j H_n^{(1)'}(\kappa_j
R)}{H_{n}^{(1)}(\kappa_j R)},
\end{equation}
and
\begin{align*}
N_{11}^{(n)}
= &\mu\left(\frac{n}{R}\right)^2\Big(\alpha_{2n}(R)-\frac{1}{R}\Big)
-\alpha_{2n}(R)\bigg[(\lambda+2\mu)\frac{\kappa_1^2
H_{n}^{(1)''}(\kappa_1 R)}{H_{n}^{(1)}(\kappa_1 R))}\\
&\qquad +(\lambda+\mu)\Big(\frac{1}{R}\alpha_{1n}(R)-\left(\frac{n}{R}
\right)^2\Big)\bigg], \\
N_{12}^{(n)}=&\mu\frac{{\rm
i}n}{R}\alpha_{1n}(R)\Big(\alpha_{2n}(R)-\frac{1}{R}\Big)
-\frac{{\rm i}n}{R} \bigg[(\lambda+2\mu)\frac{\kappa_1^2
H_{n}^{(1)''}(\kappa_1 R)}{H_{n}^{(1)}(\kappa_1 R))}\\
&\qquad +(\lambda+\mu)\Big(\frac{1}{R}\alpha_{1n}(R)-\left(\frac{n}{R}
\right)^2\Big)\bigg],\\
N_{21}^{(n)} = & -\mu\frac{{\rm
i}n}{R}\alpha_{2n}(R)\Big(\alpha_{1n}(R)-\frac{1}{R}\Big)
+\mu\frac{{\rm i}n}{R}\frac{\kappa_2^2
H_{n}^{(1)''}(\kappa_2 R)}{H_{n}^{(1)}(\kappa_2 R)},\\
N_{22}^{(n)} = & \mu\left(\frac{n}{R}\right)^2
\Big(\alpha_{1n}(R)-\frac{1}{R}\Big)
-\mu\alpha_{1n}(R)\frac{\kappa_2^2
H_{n}^{(1)''}(\kappa_2 R)}{H_{n}^{(1)}(\kappa_2 R)}.
\end{align*}

The matrix entries $N_{ij}^{(n)}, i,j=1, 2$ can be further simplied. Recall that
the Hankel function $H_n^{(1)}(z)$ satisfies the Bessel differential equation
\[
z^2 H_n^{(1)''}(z)+z H_n^{(1)'} (z)+(z^2-n^2) H_{n}^{(1)}(z)=0.
\]
We have from straighforward calculations that 
\begin{align*}
N_{11}^{(n)} &=-\alpha_{2n}(R)\Bigg[(\lambda+2\mu)\bigg[-\frac{1}{R^2}
\Big(\kappa_j R \frac{H^{(1)'}_n(\kappa_j
R)}{H^{(1)}_n(\kappa_j R)}+\left((\kappa_j R)^2-n^2\right)\Big)\bigg]\\
&\qquad+(\lambda+\mu)\left(\frac{1}{R}\alpha_{1n}(R)-\left(\frac{n}{R}
\right)^2\right)\Bigg] +\mu\left(\frac{n}{R}\right)^2\left(\alpha_{2n}
(R)-\frac{1}{R}\right)\\
&=-\alpha_{2n}(R)\Bigg[-\left(\frac{\lambda+2\mu}{R}\right)\alpha_{1n}
(R)-(\lambda+2\mu)\kappa_1^2+(\lambda+2\mu)\left(\frac{n}{R}\right)^2
+\left(\frac{\lambda+\mu}{R}\right)\alpha_{1n}(R)\\
&\qquad -(\lambda+\mu)\left(\frac{n}{R}\right)^2\Bigg]
+\mu\left(\frac{n}{R}\right)^2\left(\alpha_{2n}(R)-\frac{1}{R}\right)\\	
&=-\frac{\mu}{R}\left[\left(\frac{n}{R}\right)^2-\alpha_{1n}(R)\alpha_{2n}
(R)\right]+\alpha_{2n}(R)\omega^2\\
&= -\frac{\mu}{R}\Lambda_n(R)+\alpha_{2n}(R)\omega^2,
\end{align*}
\begin{align*}	
N_{12}^{(n)} &= -\frac{{\rm
i}n}{R}\Bigg[(\lambda+2\mu)\bigg[-\frac{1}{R^2}\Big(\kappa_j R
\frac{H^{(1)'}_n(\kappa_j R)}{H^{(1)}_n(\kappa_j
R)}+\left((\kappa_j R)^2-n^2\right)\Big)\bigg]\\	
&\qquad+(\lambda+\mu)\left(\frac{1}{R}\alpha_{1n}(R)-\left(\frac{n}{R}
\right)^2\right)\Bigg] +\frac{{\rm
i}n\mu}{R}\alpha_{1n}(R)\alpha_{2n}(R)-\mu\frac{{\rm
i}n}{R^2}\alpha_{1n}(R)\\
&= -\frac{{\rm
i}n}{R}\left[-\frac{\mu}{R}\alpha_{1n}(R)+\mu\left(\frac{n}{R}\right)^2
-(\lambda+2\mu)\kappa_1^2\right]
+\frac{{\rm i}n\mu}{R}\alpha_{1n}(R)\alpha_{2n}(R)-\frac{{\rm
i}n}{R^2}\mu\alpha_{1n}(R)\\
&= -\frac{{\rm i}n\mu}{R}\Lambda_n(R)+\frac{{\rm
i}n}{R}\omega^2,
\end{align*}
\begin{align*}
N_{21}^{(n)} &= -\mu\frac{{\rm
i}n}{R}\alpha_{2n}(R)\alpha_{1n}(R)+\frac{{\rm
i}n\mu}{R^2}\alpha_{2n}(R) +\mu\frac{{\rm
i}n}{R}\left(\frac{-1}{R^2}\right)\left(R\alpha_{2n}(R)+(\kappa_2 R)^2-
n^2\right)\\
&= -\mu\frac{{\rm i}n}{R}\alpha_{1n}(R)\alpha_{2n}(R)+\frac{{\rm
i}n\mu}{R^2}\alpha_{2n}(R) -\mu\frac{{\rm
i}n}{R^2}\alpha_{2n}(R)-\frac{{\rm i}n\mu}{R}\kappa_2^2+{\rm
i}\mu\left(\frac{n}{R}\right)^3\\
&= \frac{{\rm i}\mu n}{R}\Lambda_n(R)-\frac{{\rm
i}n}{R}\omega^2,
\end{align*}
\begin{align*}
N_{22}^{(n)} &=
\mu\left(\frac{n}{R}\right)^2\alpha_{1n}(R)-\frac{\mu}{R}\left(\frac{n}{R}
\right)^2-\mu\alpha_{1n}(R)\frac{-1}{R^2}\left(R\alpha_{2n}(R)+(\kappa_2
R)^2-n^2\right)\\
&=\mu\left(\frac{n}{R}\right)^2\alpha_{1n}(R)-\frac{\mu}{R}\left(\frac{n}{R
}\right)^2+\frac{\mu}{R}\alpha_{1n}(R)\alpha_{2n}(R)
+\alpha_{1n}(R)\mu\kappa_2^2-\mu\left(\frac{n}{R}\right)^2\alpha_{1n}(R)\\
&=-\frac{\mu}{R}\left(\left(\frac{n}{R}\right)^2-\alpha_{1n}(R)\alpha_{2n}
(R)\right)+\alpha_{1n}(R)\omega^2\\
&=-\frac{\mu}{R}\Lambda_n(R)+\alpha_{1n}(R)\omega^2.
\end{align*}
Substituting the above into \eqref{utbc}, we obtain
\begin{align*}
\mathscr B\boldsymbol{u} = \mathscr{T}\boldsymbol{u} 
=&\sum\limits_{n\in\mathbb{Z}}\frac{1}{\Lambda_n}\bigg\{
\Big[\Big(-\frac{\mu}{R}\Lambda_n(R)+\alpha_{2n}(R)\omega^2\Big)
u_n^{r}(R)\\
&+\Big(-\frac{{\rm i}n\mu}{R}\Lambda_n(R)+\frac{{\rm i}n}{R}\omega^2\Big)
u_n^{\theta}(R)\Big] \boldsymbol{e}_r + \Big[\Big(\frac{{\rm i}\mu
n}{R}\Lambda_n(R)-\frac{{\rm i}n}{R}\omega^2 \Big)u_n^{r}(R)\\
&+\Big(-\frac{\mu}{R}\Lambda_n(R)+\alpha_{1n}(R)\omega^2
\Big)u_n^{\theta}(R)\Big]\boldsymbol{e}_{\theta} \bigg\}
e^{{\rm i}n\theta}.
\end{align*}

\begin{lemma}\label{Lamn}
Let $z>0$. For sufficiently large $|n|$, $\Lambda_n(z)$ admits the following
asymptotic property
\[
\Lambda_n(z)=\frac{1}{2}(\kappa_1^2+\kappa_2^2)+\mathcal{O}\Big(\frac{1}
{|n|}\Big).
\]
\end{lemma}

\begin{proof}
 Using the asymptotic expansions of the Hankel functions \cite{W-95}
 \[ 
\frac{H_n^{(1)'}(z)}{H_n^{(1)}(z)}=-\frac{|n|}{z}+\frac{z}{2|n|}+\mathcal{O}
\Big(\frac{1} {|n|^2}\Big),
 \]
 we have 
 \[
  \alpha_{jn}(z)=\frac{\kappa_j H_n^{(1)'}(\kappa_j
z)}{H_{n}^{(1)}(\kappa_j z)}=-\frac{|n|}{z}+\frac{\kappa^2_j
z}{2|n|}+\mathcal{O}
\Big(\frac{1} {|n|^2}\Big).
 \]
A simple calcuation yields that 
\[
 \Lambda_n(z)=\left(\frac{n}{z}\right)^2-\alpha_{1n}(z)\alpha_{2n}(z)=\frac{1}{2
}(\kappa_1^2+\kappa_2^2)+\mathcal{O}\Big(\frac{1} {|n|}\Big),
\]
which completes the proof.
\end{proof}

\end{document}